%% file: OmegaCondDec24.tex
\documentclass[12pt]{article}
\usepackage{xcolor}
\definecolor{OliveGreen}{rgb}{0,0.6,0}
\definecolor{tempblue}{RGB}{36, 56, 231 }
\usepackage{cite}
\usepackage{subfigure}
\usepackage{graphicx, color, graphpap}
\usepackage{amsmath,amssymb,amsthm,bm}
\usepackage{mathabx}
\usepackage{ifthen}

\usepackage{matlab-prettifier}

\usepackage{mathrsfs}
\usepackage{mathtools}
\usepackage[shortlabels]{enumitem}

\usepackage{physics}
\usepackage{pstricks}
\usepackage{lineno}
\usepackage{float}
\usepackage{fancyhdr}
\usepackage[us,12hr]{datetime} 
\usepackage{exscale}
\usepackage{tabularx}
\usepackage{latexsym}
\usepackage{fullpage}       
\usepackage{float}
\usepackage{verbatim}
\usepackage{multirow}
\usepackage{overpic}

\usepackage{siunitx}
\usepackage{colortbl}

\definecolor{tablegray}{RGB}{215, 219, 221 }

\usepackage{booktabs}
\usepackage{makeidx}
\makeindex
\usepackage[algosection,ruled,lined,linesnumbered,longend]{algorithm2e}

\makeindex

\usepackage{csvsimple}

\usepackage[pagebackref=true]{hyperref}
\hypersetup{
	plainpages=false,       
	unicode=false,          
	pdftoolbar=true,        
	pdfmenubar=true,        
	pdffitwindow=false,     
	pdfstartview={FitH},    
	pdftitle={
$\omega$-Condition Number and Applications 
	           },    
	pdfnewwindow=true,      
	colorlinks=true,        
	linkcolor=blue,         
	citecolor=green,        
	filecolor=magenta,      
	urlcolor=cyan           
}

\numberwithin{equation}{section}  
\numberwithin{table}{section}
\numberwithin{figure}{section}

\usepackage[noabbrev]{cleveref}

\usepackage{mathtools}

\usepackage{refcount} 

\def\R{\mathbb{R}}

\def\Rpp{\R_{++}}
\def\Sc{\mathbb{S}}

\def\Sn{\Sc^n}

\def\Snp{\Sc_+^n}

\def\Snpp{\Sc_{++}^n}

\def\Rk{\mathbb{R}^k}

\def\Rn{\mathbb{R}^n}

\def\Rnp{\mathbb{R}_+^n}
\def\Rnpp{\mathbb{R}_{++}^n}

\def\cref#1{{\normalfont(\ref{#1})}}

\newcommand*\colvec[1]{\begin{pmatrix}#1\end{pmatrix}}

\newenvironment{noteH}{\begin{quote}\small\sf \color{red} HenryW $\clubsuit$~}{\end{quote}}

\newenvironment{noteD}{\begin{quote}\small\sf \color{purple} David
$\diamondsuit$~}{\end{quote}}
\newenvironment{noteL}{\begin{quote}\small\sf \color{cyan} Leo
$\diamondsuit$~}{\end{quote}}
	
\def\cref#1{{\normalfont(\ref{#1})}}

\newtheorem{theorem}{Theorem}[section]

\newtheorem{example}[theorem]{Example}

\newtheorem{prop}[theorem]{Proposition}

\newtheorem{corollary}[theorem]{Corollary}

\newtheorem{remark}[theorem]{Remark}

\newtheorem{lemma}[theorem]{Lemma}

\newtheorem{fact}[theorem]{Fact}

\crefname{thm}{Theorem}{Theorems}
\Crefname{thm}{Theorem}{Theorems}
\crefname{problem}{Problem}{Theorems}
\Crefname{problem}{Problem}{Theorems}
\Crefname{assump}{Assumption}{Theorems}
\crefname{assump}{Assumption}{Theorems}
\crefname{conjecture}{Conjecture}{Theorems}
\Crefname{conjecture}{Conjecture}{Theorems}
\crefname{prop}{Proposition}{Propositions}
\Crefname{prop}{Proposition}{Propositions}
\crefname{cor}{Corollary}{Corollaries}
\Crefname{cor}{Corollary}{Corollaries}
\crefname{lem}{Lemma}{Lemmas}
\Crefname{lem}{Lemma}{Lemmas}
\theoremstyle{definition}
\crefname{defn}{definition}{definitions}
\Crefname{defn}{Definition}{Definitions}
\crefname{conj}{Conjecture}{Conjectures}
\Crefname{conj}{Conjecture}{Conjectures}
\crefname{remark}{Remark}{Remarks}
\Crefname{remark}{Remark}{Remarks}
\crefname{rmk}{Remark}{Remarks}
\Crefname{rmk}{Remark}{Remarks}
\crefname{example}{Example}{Examples}
\Crefname{example}{Example}{Examples}
\crefname{algorithm}{Algorithm}{Algorithms}
\Crefname{algorithm}{Algorithm}{Algorithms}
\crefname{align}{}{}
\Crefname{align}{}{}
\crefname{equation}{}{}
\Crefname{equation}{}{}

\newcommand{\textdef}[1]{\textit{#1}\index{#1}}

\newcommand{\bd}{\bar{d}}
\newcommand{\balpha}{\bar{\alpha}}

\newcommand{\cI}{{\mathcal I} }

\newcommand{\cD}{{\mathcal D} }

\newcommand{\Sk}{{\mathcal S^{k}}\,}

\newcommand{\A}{{\mathcal A}}

\newcommand{\bbm}{\begin{bmatrix}}
\newcommand{\ebm}{\end{bmatrix}}
\newcommand{\bem}{\begin{pmatrix}}
\newcommand{\eem}{\end{pmatrix}}
\newcommand{\beq}{\begin{equation}}
\newcommand{\beqs}{\begin{equation*}}
\newcommand{\bet}{\begin{table}}
\newcommand{\eeq}{\end{equation}}
\newcommand{\eeqs}{\end{equation*}}
\newcommand{\beqr}{\begin{eqnarray}}

\newcommand{\omegatwo}{\omega^{-2}}

\newcommand{\dbardhat}{\begin{pmatrix}\bar d\cr \hat d\end{pmatrix}}

\DeclareMathOperator{\size}{{size}}

\DeclareMathOperator{\cond}{{cond}}
\DeclareMathOperator{\norms}{{norms}}
\DeclareMathOperator{\normst}{{norms^2}}

\DeclareMathOperator{\normsa}{{norms^\alpha}}
\DeclareMathOperator{\normsti}{{norms^{-2}}}

\DeclareMathOperator{\nul}{null}
\DeclareMathOperator{\range}{range}

\DeclareMathOperator{\dist}{dist}
\DeclareMathOperator{\kvec}{{vec}}

\DeclareMathOperator{\adj}{{adj}}

\DeclareMathOperator{\blkdiag}{{blkdiag}}
\DeclareMathOperator{\diag}{{diag}}

\DeclareMathOperator{\DWT}{{DWT}}
\DeclareMathOperator{\Triu}{{Triu}}
\DeclareMathOperator{\triu}{{triu}}
\DeclareMathOperator{\Triuk}{{Triu_k}}
\DeclareMathOperator{\triuk}{{triu_k}}
\DeclareMathOperator{\Trirk}{{Trir_k}}
\DeclareMathOperator{\trirk}{{trir_k}}
\DeclareMathOperator{\Dt}{D_{+tk}}

\DeclareMathOperator{\TDiag}{{Diags_2}}
\DeclareMathOperator{\tdiag}{{diags_2}}
\DeclareMathOperator{\Diag}{{Diag}}

\DeclareMathOperator{\Ddiag}{{Ddiag}}

\newcommand{\nc}{\newcommand}
\nc{\arrow}{{\rm arrow\,}}
\nc{\Arrow}{{\rm Arrow\,}}
\nc{\BoDiag}{{\rm B^0Diag\,}}
\nc{\bodiag}{{\rm b^0diag\,}}

\nc{\Mm}{{\mathcal M}^{m} }
\nc{\Mmn}{{\mathcal M}^{mn} }
\nc{\Mnr}{{\mathcal M}_{nr} }
\nc{\Mnmr}{{\mathcal M}_{(n-1)r} }
\nc{\kwqqp}{Q{$^2$}P\,}
\nc{\kwqqps}{Q{$^2$}Ps}

\nc{\notinaho}{(X,S)\in \overline{AHO}(\A)}
\nc{\inaho}{(X,S)\in AHO(\A)}

\newcommand{\bea}{\begin{eqnarray}}%
\newcommand{\eea}{\end{eqnarray}}%
\newcommand{\beas}{\begin{eqnarray*}}%
\newcommand{\eeas}{\end{eqnarray*}}%
\newcommand{\Rmn}{\R^{m \times n}}%
\newcommand{\Rnn}{\R^{n \times n}}%

{}

\newcommand{\Hnp}[1][]{\,\mathbb{H}_+^{\ifthenelse{\equal{#1}{}}{n}{#1}}}
\newcommand{\Hn}[1][]{\,\mathbb{H}^{\ifthenelse{\equal{#1}{}}{n}{#1}}}
\newcommand{\Hk}[1][]{\,\mathbb{H}^{\ifthenelse{\equal{#1}{}}{k}{#1}}}
\newcommand{\Dn}[1][]{\,\mathbb{D}^{\ifthenelse{\equal{#1}{}}{n}{#1}}}

\DeclareMathOperator*{\argmin}{argmin}

\newcommand{\Id}{\ensuremath{\operatorname{Id}}}

\title{
\href{http://orion.math.uwaterloo.ca/~hwolkowi/henry/reports/ABSTRACTS.html}{
The $\omega$-Condition Number:\\
Applications to 
Optimal Preconditioning\\
and\\
Low Rank Generalized Jacobian Updating
}
   \footnote{
Emails resp.: w2jung@uwaterloo.ca, david.torregrosa@ua.es,
hwolkowicz@uwaterloo.ca}
   \footnote{
   This report is available at URL:
   \href{https://www.math.uwaterloo.ca/\~hwolkowi/henry/reports/ABSTRACTS.html}{www.math.uwaterloo.ca/\~{ }hwolkowi/henry/reports/ABSTRACTS.html}
   }
}

\author{
\href{https://uwaterloo.ca/combinatorics-and-optimization/about/people/group/50}{%
Woosuk L. Jung}\thanks{
\href{http://www.math.uwaterloo.ca/co/}{Department of Combinatorics and Optimization}, University of Waterloo, 200 University Avenue West, Waterloo, ON N2L 3G1, Canada}
 \and
\href{https://cvnet.cpd.ua.es/curriculum-breve/es/torregrosa-belen-david/110216}{%
David Torregrosa-Bel\'en}\thanks{
\href{https://www.cmm.uchile.cl/}{Centro de Modelamiento Matemático}, Centro de Modelamiento Matemático (CNRS IRL2807), Universidad de Chile, Beauchef 851, Santiago, Chile} \and
\href{https://www.math.uwaterloo.ca/~hwolkowi/}
{Henry Wolkowicz}\textsuperscript{$\ddagger$}
}

\date{
Revising as of \today, \currenttime
}

\begin{document}
\maketitle



\tableofcontents

\listoftables
\listoffigures

\begin{abstract}

Preconditioning is essential in iterative methods for solving linear
systems. It is also the implicit objective in updating 
approximations of Jacobians in optimization methods, e.g.,~in 
quasi-Newton methods. Motivated by the latter,
we study a nonclassic matrix condition number, the
\textdef{$\omega$-condition number}, $\omega$ for short. 
$\omega$ is the ratio of the arithmetic and
geometric means of the singular values, rather than largest and smallest.
Moreover, unlike the latter classical $\kappa$ condition number, 
$\omega$ is \emph{not} invariant under
inversion, an important point that allows one to recall that it is the
conditioning of the inverse that is important.

Our study is in the context of optimal conditioning for: (i) low rank updating of generalized Jacobians
arising in the context of nonsmooth Newton methods; and
(ii) iterative methods for linear systems: 
(iia) clustering of eigenvalues; (iib) convergence rates;
and (iic) estimating the actual condition of a linear system.
We emphasize that the simple functions in $\omega$ allow one
to exploit optimality conditions
and derive \emph{explicit} formulae for $\omega$-optimal preconditioners 
of special structure. Connections to partial Cholesky type
sparse preconditioners are made that modify the iterates
of Cholesky decomposition by including the entire diagonal at each iteration.
Our results confirm the efficacy of using the 
$\omega$-condition number compared to the classical 
$\kappa$-condition number.
\end{abstract}

\section{Introduction}
\label{sec:intro}

Preconditioning is essential in iterative and direct solutions of 
linear systems e.g.,~\cite{MR2169217}.
It is also the implicit objective in low rank updating of approximate
Jacobians in optimization, e.g.,~in quasi-Newton methods~\cite{DeWo:90}.
In this paper we study the $\omega$-condition number (abbreviated as
$\omega$), a nonclassic
matrix condition number that, for a positive definite matrix, is the
ratio of the arithmetic and geometric means of the eigenvalues, rather
than the largest and smallest eigenvalues of the classical
$\kappa$-condition number (abbreviated as $\kappa$). We emphasize that $\omega$ provides a more
average indication rather than a worst case measure of conditioning.
Moreover, unlike $\kappa$, $\omega$ is \emph{not} invariant under
inversion. This important property allows one to recall 
and exploit that it is the conditioning of the inverse that is important.

In particular, our original motivation is to find well conditioned,
$\omega$-optimal, low rank updates of the positive definite
generalized Jacobian that arises
in nonsmooth Newton methods e.g.,~\cite{CensorMoursiWeamsWolk:22}.
We use optimality conditions to find \emph{explicit formulae} for these
low rank updates. As well our work includes explicit formulae for
$\omega$-optimal diagonal and sparse upper triangular preconditioners.
We see that these latter relate to a sparse incomplete Cholesky factorization,
i.e.,~we modify the iterates of the Cholesky factorization by
including the entire diagonal at each iteration.
We then illustrate both the efficiency and effectiveness
of using $\omega$ compared to
$\kappa$ when solving 
positive definite linear systems.  In particular, our empirics show 
that $\omega$ is more effective in promoting the important
property of clustering of eigenvalues.

In addition, we show that $\omega$ can be evaluated
exactly following a Cholesky or LU factorization; and that it is
a better indication of the conditioning of a problem when compared to $\kappa$.

\subsection{Background and Preliminaries} 
In numerical analysis, a condition number of a matrix $A$ is the main
tool in the study of error propagation in the problem of solving the
linear equation $Ax=b$.  The linear system $Ax=b$ is said to be
well-conditioned when $A$ has a low condition number.
In particular, in the literature $\kappa(A)$ is used
as a (worst case) measure of the conditioning of a linear system $Ax=b$,
i.e.,~how much a solution $x$, the output, will change 
with respect to changes in the right-hand side $b$, the input:
\begin{equation}
\label{eq:condArelerrelinput}
\textdef{$\cond(A)$} := \frac {\|\Delta x\|/\|x\|}{\|\Delta b\|/\|b\|},
\end{equation}
e.g.,~\cite[Sect. 1.3]{stoer1980introduction}. In general, iterative
algorithms that solve $Ax=b$ require a large number of
iterations to achieve a solution with sufficient accuracy if the problem is
not well-conditioned, i.e.,~is ill-conditioned. 
For simplicity, in this paper,  we
restrict ourselves to $A$ positive definite and so $\kappa(A) =
\lambda_1(A)/\lambda_n(A) \, (=\kappa(A^{-1}))$.

\index{$\kappa$-condition number}

In order to improve the conditioning of a problem, preconditioners are 
employed for obtaining equivalent systems with better condition number.  
For example, in~\cite{davidon1975optimally} a preconditioner that minimizes 
$\kappa$ is obtained in the Broyden family of rank-two updates. 
Also, for applications to inexact Newton methods
see~\cite{BERGAMASCHI20111863,a13040100}, where it is emphasized
that the goal is to improve the \emph{clustering of eigenvalues} around $1$.
The $\omega$-condition number in particular uses \emph{all} the
eigenvalues, rather than just the largest and smallest as in the
classical $\kappa$. 
A recent survey on preconditioning is
given in~\cite{MR4175150}. We emphasize that though many heuristics are
given, the main measure of conditioning, e.g.~\cite{MR4175150},  is 
$\kappa$.\footnote{Links:
\href{https://nhigham.com/2019/01/23/who-invented-the-matrix-condition-number/}
{Who Invented the Matrix Condition Number?} and
\href{https://blogs.mathworks.com/cleve/2017/07/17/what-is-the-condition-number-of-a-matrix/}
{What is the Condition Number of a Matrix?}
}
However, in our view, $\kappa$ has the misleading property that it is
\textdef{inverse invariant}.\footnote{
In \href{https://blogs.mathworks.com/cleve/2017/07/17/what-is-the-condition-number-of-a-matrix/}
{What is the Condition Number of a Matrix?}, the derivation for $\kappa$
for $Ax=b$ 
with $\sigma_{\max},\sigma_{\min}$ largest and smallest singular values for $A$, respectively, 
is that $\|b\|\leq \sigma_{\max} \|x\|, \| \Delta b\|\geq 
\sigma_{\min}\|\Delta x\|$.
However, one can equally argue using
$\|\Delta x\|\leq \sigma_{\max}(A^{-1}) \|\Delta b\|, \| x\|\geq \sigma_{\min}(A^{-1}) \|b\|$
 to get $\frac{ \| \Delta x\| \|b \|}{\|x\| \|\Delta b\|} \leq
\kappa(A^{-1})$.}

\label{page:mainaim}
From the discussions in~\cite{MR1311636,MR929546}, the
(\emph{worst case}) condition number for the linear system $Ax=b$ is
\begin{equation}
\label{eq:condABworst}
\begin{array}{rcl}
\cond(A,b) &:=& \lim\limits_{\epsilon\downarrow 0} 
\sup\limits_{\stackrel {\|\Delta A\|\leq \epsilon \|A\|}
     {\|\Delta b\|\leq \epsilon \|b\|}}
\frac
          {\|(A+\Delta A)^{-1}(b+\Delta b)- A^{-1}b\|}
          {\epsilon \|A^{-1}b\|}
\\&=&
\kappa(A) + \frac {\| A^{-1}\|\|b\|} {\|A^{-1}b\|} \quad
    \left(=\kappa(A^{-1}) + \frac {\| A^{-1}\|\|b\|} {\|A^{-1}b\|}\right),
\end{array}
\end{equation}
where we have added the equivalence for the inverse invariance for
emphasis. The nonstandard condition number 
$\omega$ was proposed in~\cite{DeWo:90}. Interestingly enough, the authors
show that the inverse-sized BFGS and sized DFP~\cite{oren1974self} are
obtained as optimal quasi-Newton updates with respect to this measure. 
In contrast to the worst case condition number $\kappa$, 
we argue that $\omega$ is a more average type condition number and
provides a better measure for improving conditioning. Moreover, it
distinguishes between the conditioning of $A$ and $A^{-1}$.
We illustrate that $\omega$ presents advantages with
respect to the classic $\kappa$. Both are pseudoconvex
over the open convex cone of positive definite matrices, $\Snpp$; thus
a local minimum is a global minimum.
But, $\kappa$ is differentiable if, and
only if, both largest and smallest eigenvalues are singletons,
while $\omega$ is differentiable on all of $\Snpp$. 
This facilitates obtaining explicit formulae for
optimal preconditioners and avoids expensive calculations,  
see e.g.,~\cite{DeWo:90} and \Cref{sect:basic}, 
below.\footnote{                     
Since the original version of this paper was submitted, the recent report
\cite{gao2023scalable} (and many references therein) discusses
numerical scalable algorithms for $\kappa$-optimal diagonal preconditioning.
We have added relationships to this paper in this revised version.
In particular, we present an alternative algorithm as well as illustrate
that using the $\omega$-optimal formula in the positive definite case
has relatively no cost in evaluation, and is a better preconditioner.}   
Moreover, it is expensive to evaluate the classic
condition number~\cite{MR86f:65083} as it uses both
$\|A\|,\|A^{-1}\|$. For large scale, one often uses the $\ell_1$ approximation
in~\cite{MR86f:65083}. We show
that we can find the exact value of the $\omega$-condition number when a
Cholesky or LU factorization is done. Finally, we show that
$\omega$, and particularly $\sqrt{\omega(A^{-2})}$, denoted
\textdef{$\omegatwo$}, provides 
a significantly better estimate for the true conditioning of a linear
system. (Though $\omegatwo$ is currently for theoretical
purposes only as we do not yet have an efficient way of exploiting it
without calculating $A^{-1}$.)

\index{$\omegatwo=\sqrt{\omega(A^{-2})}$}
\index{$\sqrt{\omega(A^{-2})}=\omegatwo$}

\subsection{Notation}
\label{subsect:pre}

We denote: \textdef{$\Rn$} as the real Euclidean space of dimension $n$,
and \textdef{$\Rnp, \Rnpp$} as the nonnegative and 
positive orthants, respectively;
\textdef{$\Rmn$} as the space of $m\times n$ matrices; \textdef{$\Sn$}
as the space of $n\times n$ symmetric matrices;
\textdef{$\Snp$} and \textdef{$\Snpp$} for the cone
of positive semidefinite and positive definite $n\times n$ symmetric
matrices, respectively; and $A\succeq 0$ (resp., $\succ 0$ ) as
$A$ is in $\Snp$ (respectively, $\Snpp$).
We use  the Kronecker product and Hadamard (elementwise) product
$A \otimes B,C \circ D$, respectively, with the matrix to vector
columnwise vectorization \textdef{$x= \kvec(X)$}.
\index{Kronecker product, $A\otimes B$}
\index{$A\otimes B$, Kronecker product}
\index{Hadamard product, $A\circ B$}
\index{$A\circ B$, Hadamard product}

We use \textdef{$\Diag:\R^n\to\R^{n\times n}$} to denote the linear operator that maps a vector $v$  into the diagonal matrix   $\Diag(v)$  whose diagonal is $v$.  Its adjoint  operator is  denoted by \textdef{$\diag = \Diag^*$}.

For integers $t\geq s$, we let \textdef{$[s,t] = \{s,s+1,\ldots,t\}$}.
For a positive integer $k$, let \textdef{$[k] = [1,k]$}
and denote \textdef{$t(k) = k(k+1)/2$, triangular number}.
\index{triangular number, $t(k) = k(k+1)/2$}

For a differentiable function $f:\R^n\to\R$,
we use $\nabla f$ for the gradient. 
\label{page:RnR}
If         
the dimension $n = 1$,
we just write 
$f^{\prime}$ for the derivative of $f$. Given a nonempty open set  $\Omega\subseteq\R^n$, a function $f:\Omega \to\R$ is said to be \emph{pseudoconvex} on $\Omega$ if it is differentiable and
\label{page:transposegrad}
\[
\nabla f(x)^T (y-x) \geq 0 \Longrightarrow f(y) \geq f(x), \quad
\forall x, y\in\Omega.
\]
This implies that for an open convex set $\Omega$ and a
\textdef{pseudoconvex function} $f:\Omega\to\R$,   
we have: $\nabla f(x) = 0$ is  a necessary and sufficient condition for
$x$ to be a global minimizer of $f$ in $\Omega$, see,
e.g.,~\cite{Mang:69}.

\subsection{Outline and Main Results}
The goal of this paper is to show the efficacy of using $\omega$ when
compared to $\kappa$, i.e.,~to illustrate that $\omega$ outperforms
$\kappa$ as a condition number. 
And in particular, we illustrate this on
preconditioning and low rank updating.

\Cref{sect:basic,sect:omegaoptpreconds} introduce basic and 
new properties of $\omega$ as well as \emph{new}
explicit formulae for $\omega$-optimal preconditioners of special
structure: a new $\omega$-optimal diagonal preconditioner
is given in~\Cref{thm:optdiaginvprec}; and various triangular types are
included.
Efficiency and accuracy of computing $\omega$
is given in~\Cref{sect:efficevalomega}. Indeed
\emph{the condition number of the condition number is the condition
number} holds for $\kappa$ but not for $\omega$ indicating that
numerical calculations of ill-conditioned $\kappa$ can be very
inaccurate in contrast to $\omega$.

We include connections to preserving sparsity and to incomplete Cholesky
preconditioners. 
(Further explicit formulae of $\omega$-optimal preconditioners 
with special structure are given in~\Cref{sect:furhterpreconditioners}.)

In \Cref{subsec:erroranal} we empirically illustrate that
$\omega$ is a better indicator of conditioning for iterative solutions
of linear equations. Moreover,~\Cref{rem:condofcond} provides the
justification for using \textdef{$\omegatwo:=\sqrt{\omega(A^{-2})}$} 
as a measure and emphasizing
the advantage over $\kappa$ of not being inverse invariant. This includes
empirical results for better clustering of eigenvalues, 
\Cref{fig:smoothingeigs}. Though as mentioned above, $\omegatwo$ is
currently only for theoretical purposes. 

In \Cref{sec:genJac}, we derive $\omega$-optimal conditioning for
low rank updates of positive definite matrices. These updates
often arise in the construction of generalized Jacobians.  

Numerical results are in \Cref{section:NumericalTests}. We 
use the linear equations that involve positive definite matrices
as well as the generalized Jacobians for our original motivation.
We empirically illustrate  that reducing the
$\omega$-condition number improves the performance of
iterative methods for solving these linear systems. 

Conclusions are provided in \Cref{sect:conclusion}.

\section{Properties and Optimal Preconditioning: $\omega$ vs $\kappa$}
\label{sect:numerevalomega}

We now introduce basic and new properties of $\omega$,
and study the efficiency of its numerical evaluation. 
In addition, we empirically compare its effectiveness with $\kappa$ for
preconditioning, clustering of eigenvalues, and in estimating the 
actual conditioning of positive definite linear systems.

\label{page:5preconds}
In particular, we derive  the following
explicitly found optimal $\omega$-preconditioners (scalings):
(i) diagonal~\cref{eq:optdiagscal};
(ii) block diagonal~\cref{eq:blkdiagscal};
(iii) incomplete upper triangular~\cref{eq:omegaoptincchol};
(iv) lower triangular two diagonal~\cref{eq:omegaobj};
(v) upper triangular  diagonal~\cref{eq:uppertriag2diag}. 
Specifically, preconditioners (i)-(iii) maintain their sparsity
properties after matrix inversion. We include empirical comparisons
with state-of-the-art sparse incomplete Cholesky preconditioners.

\subsection{Basic Properties and $\omega$-Optimal Diagonal Preconditioner}
\label{sect:basic}
For iterative solutions of linear systems
a preconditioner $S$ is often essential, e.g.,~for preconditioned
conjugate gradients for $Ax=b, A\succ 0$, we
solve $(S^TAS) \tilde x =\tilde b = S^T b,\, x=S\tilde x$, 
see e.g.,~\cite{MR3638573,MR2169217,golvl}. Moreover, 
it is known that the simple scaling diagonal preconditioner using
the norms of the columns of $A$ is the 
optimal diagonal preconditioner with respect to $\omega$ and
is efficient in practice, see~\cite{DeWo:90,PiniGambolati:90}, 
i.e.,~$\omega$ validates the use of this specific diagonal
preconditioner.\footnote{In \cite{gao2023scalable} the motivation for
numerically finding diagonal $\kappa$-optimal preconditioners was the
lack of theoretical validation. See also the near optimality results
in~\cite{MR40:6760}. Validation using $\omega$ is now provided in
\Cref{prop:precond}, \Cref{item:optdiagprecond}. However, an improved
diagonal preconditioner is provided in~\Cref{thm:optdiaginvprec}.
}
Various preconditioners based on (partial) factorizations of $A$, are
compared in~\cite{MR3638573}. One is the QR-factorization. We note that
scaling columns is an essential part of a QR-factorization. We see below
that our $\omega$-optimal preconditioners are related to a modified
QR-factorization (Cholesky for positive definite systems). 
Moreover, convergence rates of iterative methods are
correlated to clustering of eigenvalues of $A^TA$,
see e.g.,~\cite{MR98j:65023}. We see below in \Cref{subsec:erroranal} that the
$\omega$-optimal preconditioners promote this property better than those
for $\kappa$.

The  optimal diagonal preconditioner is extended 
to the block diagonal case in~\cite{KrukDoanW:10}.
We now summarize these and other  basic properties of
$\omega$ in the following~\Cref{prop:precond}.
We include a proof of \Cref{prop:precond}, \Cref{item:optdiagprecond},
that is different than that provided in~\cite{DeWo:90} so as to
emphasize the extension to new  formulae
for  $\omega$-optimal preconditioners  
in~\Cref{sect:omegaoptpreconds,sect:optDponeprecondij,sect:lowerTrTwoDiag}.

\begin{prop}[{\cite{DeWo:90,KrukDoanW:10}}]
\label{prop:precond}
The following statements hold.
\begin{enumerate}
\item 
\label{item:globminomega}
$\omega$ is pseudoconvex \index{pseudoconvex function} on the 
cone of symmetric positive definite matrices;
thus every stationary point is a global minimizer of $\omega$.
\item 
\label{item:optdiagprecond}
Let $V$ be a full rank $m \times n$ matrix, $n \leq m$.
Then the optimal column scaling that minimizes $\omega$ is given by:
\begin{equation}
\label{eq:optdiagscal}
d^* = (d_i^*) = \argmin_{d\in \Rnpp}  
      \omega (({V}{\Diag(d)})^{T}({V}{\Diag(d)})),
\quad d_i^* =\frac{1}{\| {V}_{:,i} \|}, \,i\in[n],
\end{equation}
\label{page:Rnpp}
where ${V}_{:,i}$ is the $i$-th column of ${V}$.
\item 
\label{item:optblockdiagprecond}
Let ${V}$ be a full rank $m \times n$ matrix, $n \leq m$, with block
structure ${V}=\begin{bmatrix} {V}_1&{V}_2&\ldots &{V}_k\end{bmatrix}$, ${V}_i \in \R^{m \times n_i}$.
Then \emph{an} optimal corresponding block diagonal scaling 
\[
{D}=\begin{bmatrix} 
        {D}_1&0&0&\ldots &0
\cr
        0& {D}_2&0&\ldots &0
\cr
        \ldots & \ldots &\ldots &\ldots &\ldots 
\cr
        0& 0&0&\ldots &{D}_k
\end{bmatrix}, \quad {D}_i \in \R^{n_i\times n_i},
\]
that minimizes the measure
$\omega$, i.e.,
\begin{equation}
\label{eq:blkdiagscal}
\min ~ \omega (({V}{D})^{T}({V}{D})),
\end{equation}
over ${D}$ block diagonal, is given by the factorization
\[
{D}_i{D}_{i}^T=\{{V}_i^T\,{V}_i\}^{- 1}, \quad i\in[k].
\]
\end{enumerate}
\end{prop}
\begin{proof}
The results are proved in~\cite{DeWo:90,KrukDoanW:10}. We provide a new
proof of \Cref{item:optdiagprecond} as it leads to different extensions
below. Moreover, this proof illustrates the ease in differentiating
	$\omega$ and applying to derive \emph{explicit} formulae.

Let $d:=\diag(D), W:= V^TV, w = \diag(W)$ and note that 
\[
\begin{array}{rcl}
 \omega(d) := \omega (({V}{\Diag(d)})^{T}({V}{\Diag(d)})) 
&=&
 \frac 1{n\det(V^TV)^{1/n}}   
           \frac {\langle w, d\circ d\rangle}
          {\det(D)^{2/n}}
\\&=:&
 K \frac {\sum_{i=1}^n w_i d_i^2}
         {\prod_{i=1}^n  d_i^{2/n}}
\\&=:&
 K \frac {f_w(d)}{g(d)},
\end{array}
\]
thus defining the constant $K>0$ and functions $f_w,g:\Rn_{++} \to \Rpp$.
The reason for including this proof is to emphasize that $V$ only
appears in the numerator $f_w$ of the function to be minimized as the
denominator involves only $d$.

We now differentiate this pseudoconvex function with respect to $d_i$ :
\[
\begin{array}{rcl}
\frac {\partial \omega(d)}{\partial d_i}
&=&
\frac K{g(d)^2}
\left( g(d)2w_id_i - f_w(d) \frac 2n g(d) \frac 1{d_i} \right) 
\\&=&
\frac {2K}{g(d)}
\left( w_id_i  - \frac 1n f_w(d)   \frac 1{d_i} \right) 
\\&=&
\frac {2K}{g(d)}
\left( \frac 1{d_i}  - \frac 1n f_w(d)   \frac 1{d_i} \right) 
\\&=&
0,
\end{array}
\]
since $w_i = \|V_{:,i}\|^2 = 1/d_i^2\implies f_w(d) = n$.
\end{proof}

We now derive properties for the (square root)
of the $\omega$-condition number of $A^{-2}$, denoted $\omegatwo$.
The motivation for this is introduced below in~\cref{eq:condAinv2}.
\begin{theorem}
\label{thm:optdiaginvprec}
Let $A\in \Snpp$ and denote the Hadamard (elementwise)
product $B:= A^{-1}\circ A^{-1}  \in  \Snpp \cap \R^{n\times n}_+$.
Let $\bar d\in \Rnpp, \bar D = \Diag(\bar d)$, be the solution of
\[
B\bar d = \diag(\bar D^{-1})>0,
\]
and let
\[
D = \bar D^{1/2}, \, d=\diag(D).
\]
Then, $\bar d^T B \bar d = n$. Moreover, $d$ provides the
$\omegatwo$-optimal scaling, i.e.,~the $\omega$-optimal diagonal scaling 
$D=\Diag(d)$ of $A$ with respect to $A^{-2}$:
\[
d = \argmin_{D = \Diag(d)\succ 0} \omega ( D A^{-1} D D A^{-1} D ).
\]
The corresponding optimal scaling (preconditioning) for solving $Ax=b$,
with respect to the motivation for using $A^{-2}$ in~\cref{eq:condAinv2}, is
\[
(D^{-1}AD^{-1})(Dx) = D^{-1}b.
\]
\end{theorem}
\begin{proof}
First note $\bar d^TB\bar d = \bar d^T \diag \bar D^{-1} = n$. 
From~\cref{eq:condAinv2}, to improve conditioning for the system
$Ax=b$, we want to decrease $\omega(A^{-2})$. We
restrict to a diagonal scaling and try to find:
\begin{equation}
\label{eq:diagscalinv}
\begin{array}{rcl}
d  
&=&
\argmin\limits_{D = \Diag(d)\succ 0} \omega ( DA^{-1} DD A^{-1} D)
\\&=&
\argmin\limits_{D = \Diag(d)\succ 0} \omega ( A^{-1} D^2 A^{-1} D^2) 
\\&=&
\argmin\limits_{\bar D = \Diag(\bar d)\succ 0} \frac
 {\frac 1n\trace ( A^{-1} \bar D A^{-1} \bar D) }
 {\det ( A^{-1} \bar D A^{-1} \bar D) ^{1/n}},\quad \bar D = D^2
\\&=&
       \argmin\limits_{\bar D = \Diag(\bar d)\succ 0} 
\frac {\det(A)^{\frac 2n}}n
\frac
          {\trace ( A^{-1} \bar D A^{-1} \bar D) }
	  {\det ( \bar D^{2/n})}.
\\&=&
       \argmin\limits_{\bar D = \Diag(\bar d)\succ 0} 
\frac
          {\trace ( A^{-1} \bar D A^{-1} \bar D) }
	  {\det ( \bar D^{2/n})}.
\end{array}
\end{equation}
We use  the Kronecker product notation and Hadamard product
notation $\otimes,\circ$, with the vectorization
$\kvec(\cdot)$,and obtain
\[
\trace  A^{-1} \bar D A^{-1} \bar D = \kvec(\bar D)^T A^{-1}\otimes
A^{-1} \kvec(\bar D)
   = \bar d^T A^{-1}\circ A^{-1}\bar d =: \bar d^TB\bar d =:\bar f(\bar d),
\]
thus defining the positive definite matrix $B$ and quadratic form $\bar
f(\bar d)$. We let 
\[
\bar g(\bar d) = \prod_{i=1}^{n}\bar d_i^{2/n} = \det(\bar D^{2/n}).
\]
Then 
	\[
	\nabla \bar f(\bar d) = 2B\bar d,   \quad 
\nabla  \bar g(\bar d) = \begin{pmatrix} \frac{2g(d)}{nd_i}\end{pmatrix} 
             = \frac{2\bar g(\bar d)}{n}\bar D^{-1} e.
	\]

From the minimization problem in~\cref{eq:diagscalinv}, we want the
stationary point
\[
\begin{array}{rcl}
0 
&=&
\frac 1{\bar g(\bar d)^2}
\left(\bar g(\bar d) \nabla \bar f(\bar d) -
	  \bar f(\bar d) \nabla \bar g(\bar d)\right)
\\&=&
2\bar g(\bar d) B\bar d -
	2\frac   {\bar f(\bar d) \bar g(\bar d)}n \bar D^{-1}e
\\&=&
B\bar d - \frac {\bar f(\bar d) }n \bar D^{-1}e.
\end{array}
\]
We normalize and get the two equations
\[
\bar f(\bar d) = n, \quad B\bar d = \diag(\bar D^{-1}).
\]
\end{proof}

\begin{remark}
Though currently only of theoretical interest due to dependence on
having $A^{-1}$, we note that 
solving for $\bar d$ in~\Cref{thm:optdiaginvprec}
can be done by e.g.,~Newton's method. If $B$ is diagonal, then an explicit 
solution is $\bar d_i := \frac 1{\sqrt {B_{ii}}}$.
Note that $B$ diagonal holds if, and only if, $A$ is diagonal and then
the optimal diagonal preconditioner is 
\[
\bar D = \Diag(\bar d) =  \sqrt{B^{-1}} = A.
\]
Therefore, the optimal preconditioner for $A$ is
$\bar D^{-1/2} = A^{-1/2}$ which agrees with our optimal $\omega$
preconditioner.
In general, with $\alpha := \bar d^TB\bar d$, then
 $\sqrt \frac n\alpha \bar d$
provides a good starting point for Newton's method, as
it is highly likely that the matrix $B$ is significantly diagonally dominant.
We solve
\[
F(d) := \Diag(d)Bd - e=0.
\]
The Jacobian with the matrix representation is
\[
F^\prime(d)(\Delta d)=  \Diag(d)B\Delta d +  \Diag(\Delta d)Bd
	  = \begin{bmatrix} \Diag(d)B+  \Diag(Bd) \end{bmatrix} (\Delta d ).
\]
In our experiments Newton's method always converged in a few
iterations, in fact $4$ iterations independent of 
$n$.\footnote{This could be a result 
of monotonicity arising from the nonnegativity of both $F,F^\prime$.
It is still an open question whether $d$ can be found efficiently
without explicitly finding $A^{-1}$ first.}

\end{remark}

We now include the gradients of the condition numbers for use in the
definitions below. In the case of $\kappa$, for simplicity and to avoid the use of subgradients, we assume that the largest and smallest eigenvalues are singletons.
\begin{lemma}
\label{lem:gradconds}
Let $A\in \Snpp$ with eigenvalues $\lambda _1 \geq \lambda_2\geq \ldots
\geq \lambda_{n-1} \geq \lambda_n$, with corresponding orthonormal
eigenvectors $v_1,\ldots,v_n$. Then:
\begin{enumerate}
\item 
\[
\begin{array}{rcl}
\nabla \omega(A) &=& \frac 1{n \det(A)^{1/n}}
\left( I - \frac {\trace A}n  A^{-1}\right)
\, \text{ is indefinite},
\\&& \qquad
\text{  with  } \|\nabla \omega(A)\| = \frac 1{n \det(A)^{1/n}}
\max \left\{
1- \frac {\trace A}{n\lambda_1},
\frac {\trace A}{n\lambda_n}-1
\right\}.
\end{array}
\]
\item In addition, assume that the largest and smallest eigenvaules of $A$ are singletons, i.e., $\lambda _1 > \lambda_2\geq \ldots
\geq \lambda_{n-1} > \lambda_n$. Then:
\[
\nabla \kappa(A) = \frac 1{\lambda_n}
\left( v_1v_1^T - \kappa(A) v_nv_n^T\right),  
\, \text{ is indefinite},\text{   with  }
\|\nabla \kappa(A)\| =  \frac {\kappa(A)}{\lambda_n}.
\]
\end{enumerate}
\end{lemma}
\begin{proof}
\begin{enumerate}
\item
The gradient is
\[
\begin{array}{rcl}
\nabla \omega(A)
&=&
 \frac 1{n \det(A)^{2/n}}
\left(\det(A)^{1/n} I - \frac {\trace A}n  
  \det(A)^{\frac 1n - 1}\adj A\right)\succ 0;
\\&=&
 \frac 1{n \det(A)^{2/n}}
\left(\det(A)^{1/n} I - \frac {\trace A}n  
  \det(A)^{\frac 1n - 1}\adj A\right)\succ 0;
\\&=&
 \frac 1{n \det(A)^{1/n}}
\left( I - \frac {\trace A}n  A^{-1}\right),
\end{array}
\]
where $\adj A$ is the adjunct, the matrix of cofactors. The last
expression follows from $A^{-1} = \frac 1{\det(A)}\adj A$.
The  indefiniteness and norm follow from:
\[
\begin{array}{rcll}
\lambda_{\max}(I-\frac {\trace A}n A^{-1}) 
& = &
\max_{\|x\|=1} x^T(I-\frac {\trace A}n A^{-1}) x
\\& = &
1-  \frac {\trace A}n\min_{\|x\|=1}    x^TA^{-1} x
\\& = &
1- \frac {\trace A}{n\lambda_1}, & \text{with attainment at  } x = v_1,
\\& > & 0;
\end{array}
\]
\[
\begin{array}{rcll}
\lambda_{\min}(I-\frac {\trace A}n A^{-1}) 
& = &
\min_{\|x\|=1} x^T(I-\frac {\trace A}n A^{-1}) x
\\& = &
1+  \frac {\trace A}n\min_{\|x\|=1}  ( - x^TA^{-1} x)
\\& = &
1-  \frac {\trace A}n\max_{\|x\|=1}    x^TA^{-1} x
\\& = &
1- \frac {\trace A}{n\lambda_n}, & \text{with attainement at  } x = v_n,
\\&<& 0.
\end{array}
\]
\item
Since the eigenvalues are singletons, they are differentiable with
gradients $v_1v_1^T, v_nv_n^T$, respectively. The result follows from
the definitions of the gradient of the fractional function $\kappa$,
the spectral norm, and orthonormality of the eigenvectors.
\end{enumerate}
\end{proof}

\subsubsection{Efficiency and Accuracy for Evaluating $\omega,\kappa$}
\label{sect:efficevalomega}
Since eigenvalue decompositions can be expensive,  
one issue with $\kappa({A})$ is how to estimate it 
efficiently when the size of matrix ${A}$ is large.
A survey of estimates and, in particular, estimates using
the $\ell_1$-norm, is given in \cite{MR86f:65083,Hig:87}. 
Extensions to sparse matrices and block-oriented generalizations are 
given in \cite{grle:81,MR1780268}. 
Results from these papers form the basis of the 
\href{http://www.mathworks.com/help/techdoc/ref/condest.html}{\emph{condest}}
command in \textsc{Matlab}. 
More recently~\cite{gao2023scalable} 
deals with scalable methods for finding the $\kappa$-optimal diagonal
preconditioner.
This illustrates the difficulty in accurately estimating $\kappa(A)$.

On the other hand,
the measure $\omega({A})$ can be calculated using the trace and determinant
function that do not require eigenvalue decompositions. Derivatives are
given in \Cref{lem:gradconds} above.\footnote{Since
the first version of this paper we have been made aware of the
new CVX \textsc{Matlab} function  
\href{https://github.com/cvxr/CVX/blob/master/functions/det_rootn.m}
{\textdef{$det\_rootn$}} that calculates
$\det(A)^{1/n}$, the denominator of $\omega$, using the Cholesky decomposition.}
However, for large $n$, the determinant is also numerically
difficult to compute as it could easily result in an overflow
$+\infty$ or $0$ due to the limits of finite precision arithmetic,
e.g.,~if the order of $A$ is $n=50$ and the eigenvalues 
$\lambda_i=.5, \forall i$, then the determinant $.5^n$ is zero to
machine precision. A similar problem arises for e.g.,~$\lambda_i = 2,
\forall i$ with overflow.
In order to overcome this problem, we take the $n$-th root first and
then the product, i.e.,~we define the value obtained from the spectral
factorization as
\[
\textdef{$\omega_{\rm eig}(A)$} = \frac 
{\sum_{i=1}^n\lambda_i(A)/n}{\prod_{i=1}^n(\lambda_i(A)^{1/n})}.
\]
We now let $A=R^TR=LUP$ denote the Cholesky and $LU$ factorizations,
respectively, with appropriate permutation matrix $P$.
We assume that $L$ is unit lower triangular. Therefore,
\begin{equation}
\label{eq:detchol}
\det(A)^{1/n} = \det(R^TR)^{1/n} = \det(R)^{2/n} = 
\prod_{i=1}^n\left(R_{ii}^{2/n}\right).
\end{equation}
Similarly, 
\begin{equation}
\label{eq:detlu}
\det(A)^{1/n} = \det(LUP)^{1/n} = 
\prod_{i=1}^n\left(\left|U_{ii}\right|^{1/n}\right).
\end{equation}
Therefore, we find $\omega(A)$ with numerator $\trace(A)/n$ and denominator
given in \cref{eq:detchol} and \cref{eq:detlu}, respectively:
\[
\textdef{$\omega_R(A)$} = \frac 
{\trace(A)/n}
{\prod_{i=1}^n\left(R_{ii}^{2/n}\right)}, \quad
\textdef{$\omega_{LU}(A)$} = \frac 
{\trace(A)/n}
{\prod_{i=1}^n\left(\left|U_{ii}\right|^{1/n}\right)}.
\]

\Cref{table:wcondtimes,table:wcondprec} provide comparisons
on the time and precision from the three different factorization methods. Each
column presents different order of $\kappa$-condition number, while each
row corresponds to different decompositions with different size $n$ of
the problem. 
We form the random matrix using $A=QDQ^T$ for random orthogonal $Q$ and
positive definite diagonal $D$. We then symmetrize $A \leftarrow
(A+A^T)/2$ to avoid roundoff error in the multiplications.
Therefore, we consider the evaluation using $D$ as the \emph{exact
value} of $\omega(A)$, i.e.,
\[
\omega(A) = \frac{\sum_{i=1}^n\left( D_{ii}\right) /n} {\prod_{i=1}^n \left( D_{ii}^{1/n}\right)}.
\]
\Cref{table:wcondprec} shows
the absolute value of the difference between the exact $\omega$-condition number and the
$\omega$-condition numbers obtained by making use of each factorization,
namely, $\omega_{\rm eig}$,  $\omega_{R}$ and $\omega_{LU}$.
Surprisingly, we see that both the Cholesky and LU decompositions give
better results than the eigenvalue decomposition.
\begin{table}[h]
\resizebox{\columnwidth}{!}{%
	\input{tableomegatimes.tex}}
\caption{CPU sec. for evaluating $\omega(A)$, 
	averaged over the same 10 random instances; 
	eig, R, LU are eigenvalue, Cholesky, LU decompositions, respectively.
}
\label{table:wcondtimes}
\end{table}
\begin{table}[h]
\resizebox{\columnwidth}{!}{%
	\input{tableomegaprec.tex}}
\caption{Precision of evaluation of $\omega(A)$
	averaged over the same 10 random instances.
	eig, R, LU are eigenvalue, Cholesky, LU decompositions, respectively.
}
\label{table:wcondprec}
\end{table}

Moreover, \Cref{fig:compareaccuracy} illustrates a comparison of
accuracy in evaluations of $\omega,\kappa$. We use one positive
definite matrix with spectral decomposition $A = QDQ^T$, and $n =
1000, density = 1e-4$ with $\kappa = 200$. We use perturbations
of the eigenvalues $\|D(\epsilon) - D\|/\|D\|=1e-8$ and reform
$A(\epsilon) = QD(\epsilon)Q^T$.  \Cref{fig:compareaccuracy} clearly shows 
that $\omega$ is calculated more accurately as the ill-conditioning grows.
This relates to the \emph{condition number of the condition number} in
\Cref{rem:condofcond}.

\begin{figure}[h]\centering
\includegraphics[width = 0.5\textwidth]{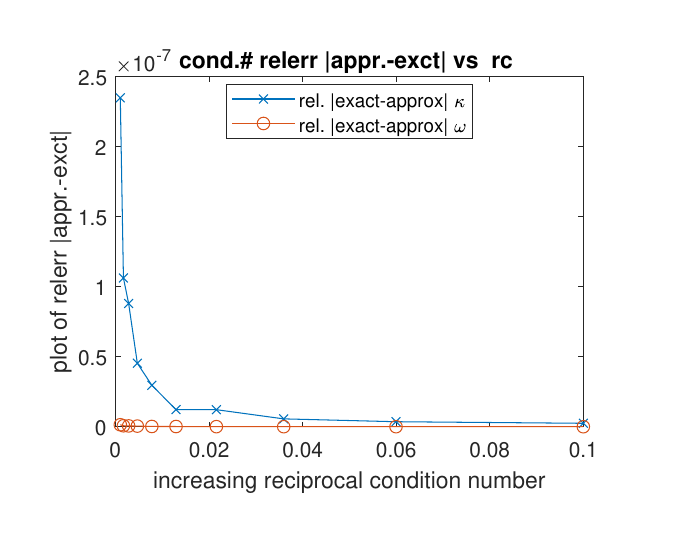}\hfill
\includegraphics[width = 0.5\textwidth]{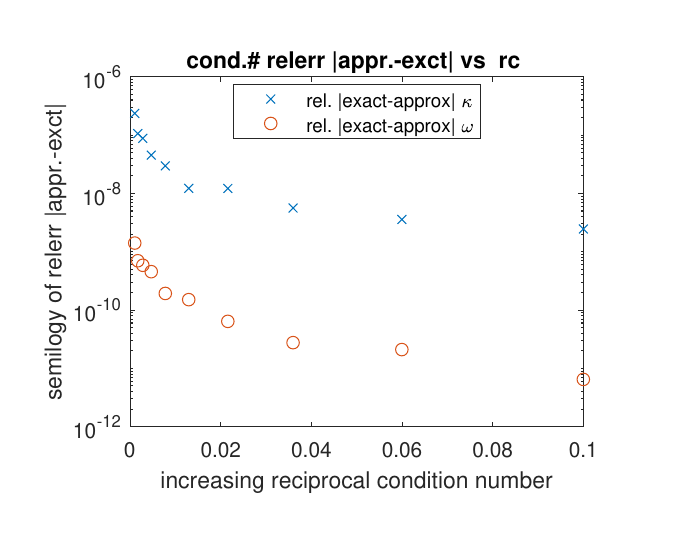}
\caption{Comparing accuracy under perturbations for calculating
	$\omega,\kappa$.
}
\label{fig:compareaccuracy}
\end{figure}
\begin{remark}
\label{rem:condofcond}
Moreover, if we consider $b$ as the input to a function $G$ with output
$x$, then a Taylor type argument gives to first order condition number
as in~\cref{eq:relresratio}
\begin{equation}
	\label{eq:condofcond}
	\cond(G) = \frac {\|b\|}{\|G(b)\|} \frac {\|\Delta G\|}{\|\Delta b\|}
	\cong
	\| \nabla G(b)\| \frac {\|b\|}{\|G(b)\|},
\end{equation}
a first order approximation for the condition number of $G$. Therefore,
if $G$ is one of $\kappa,\omega$, we get the \textdef{condition number
of the condition number}, see e.g.,~\cite{MR1311636,MR88i:15014} and 
the related result that for $\kappa$,
\label{page:condofcond}
\textdef{the condition number of the condition number is the
	 condition number}.  
We have observed empirically that the condition number of $\omega$ is
significantly smaller than the condition number of $\kappa$.

Let $G(\cdot) := A^{-1}(\cdot)$.
Let $v_i$ be orthonormal eigenvectors of $A$ and
$b = \sum \beta_iv_i, \Delta b = \sum \delta_iv_i$
and by abuse of notation 
\[
\|G(b)\|^2 = \langle \beta^2,\frac 1{\lambda^2}\rangle = \sum_i 
         \frac {\beta_i^2}{\lambda_i^2},\qquad
\|G(\Delta b)\|^2 = \langle \delta^2,\frac 1{\lambda^2}\rangle=
         \sum_i \frac {\delta_i^2}{\lambda_i^2},
\]
We now consider the conditioning for the linear system $Ax=b$.
From \cref{eq:condArelerrelinput} (squared) we have:
\begin{equation}
\label{eq:condAinv2}
\begin{array}{rcl}
\cond(A)^2
&=&
 \frac {\|b\|^2}{\|G(b)\|^2} \frac {\|G(\Delta b)\|^2}{\|\Delta b\|^2}
\\&=&
 \frac {\|b\|^2}{\|\Delta b\|^2} 
    \frac {\langle \delta^2, \frac 1{\lambda^2}\rangle }
	 {\langle \beta^2,  \frac 1{\lambda^2}\rangle }
\\&\approx &
 \frac {\|b\|^2}{\|\Delta b\|^2} 
	\frac {\| \delta\|^2 \trace A^{-2} }
	 {\| \beta\|^2  \prod_i \frac1{\lambda_i^{\frac 2n}}}
\\&=&
	 \omega(A^{-2})
\end{array}
\end{equation}
where the extended/strengthened AGM with an expected value gives the
last approximation, i.e. we divide numerator and denominator by $n$ and
then apply the generalized AGM inequality, e.g.,~\cite[page 6]{MR2910302}.

For example, to make $\omega(A^{-2})$ small we can use a diagonal
scaling on left and right 
\[
\omega(DA^{-1}DDA^{-1}D).
\]

The above suggests that we should use  the 
measure $\sqrt {\omega (A^{-2})}$
rather than $\omega(A)$. the numerical tests appear to confirm this as
well. Moreover, from \cite[Prop. 2.1 (i)]{DeWo:90} and the inverse
invariance  of $\kappa$ we have
\[
1\leq \sqrt {\omega (A^{-2})} \leq \sqrt{\kappa(A^{-2})} =\kappa(A)
\leq  \sqrt{4\omega (A^{-2})} =  2\sqrt{\omega (A^{-2})},
\]
i.e.,~the measure $\sqrt {\omega (A^{-2})}$ is a valid condition number.


\end{remark}

\subsection{Error Analysis for Linear System $Ax=b$}
\label{subsec:erroranal}
\label{page:interested}
We consider the linear system $Ax=b, A\in \Snpp, b\in \Rn$. We are interested in understanding how small changes in the data affect the solution of the
system. Let $x + \Delta x$ be a solution of the perturbed system
\begin{equation}\label{eq:perturbedsystem}
	A(x + \Delta x) = b+\Delta b,
\end{equation}
where $\Delta x , \Delta b \in\Rn$. The \emph{condition number}
aims to be a measure on how strongly a relative error in the data affects the relative
error in the solution~\cite{stoer1980introduction}.  Therefore, it can be estimated as
the ratio
\begin{equation}
	\label{eq:relresratio}
	\cond : \approx \frac{ \| \Delta x\| \|b \|}{\|x\| \|\Delta b\|}\quad
	\text{ (rel. error output/rel. error input)}.
\end{equation}
Note that the above ratio depends on the choice of the perturbation
$\Delta b$, as well as on the choice of the norms.
$\kappa=\kappa(A)= \lambda_{\max}(A)/\lambda_{\min}(A)$ 
is taken as a \emph{worst
case} estimator of the condition number as the inequality
\[
cond = 
 \frac{ \| \Delta x\| \|b \|}{\|x\| \|\Delta b\|} \leq \kappa(A) 
\qquad \left(\leq 4\omega(A)^n \text{ \cite{DeWo:90}}\right),
\]
holds for all $\Delta b\in \Rn$. See \cref{eq:condABworst} above, 
and e.g.,~\cite{MR86f:65083} or ~\cite[Chapter~7]{strang2012linear} 
for further details. 

In the literature, the system is termed \emph{ill-conditioned}
if $\kappa$ is \emph{large}, and termed
 \emph{well-conditioned} otherwise.\footnote{                          
	$\kappa(A)$ is also used to measure error that arises from
	perturbations in $A$:
	$
	\frac{\|\Delta x \|}{\|x + \Delta x\|} \leq  \kappa(A) \frac{\|\Delta A\|}{\| A \|}.
	$
	The results are essentially equivalent.}
We now study the correlation between the three estimates of cond: 
(i) $\kappa$, (ii) $\omega$, and (iii)$\omegatwo$, with the estimate of $\cond$ 
evaluated by sampling. We sample as follows:
\begin{enumerate}
	\item
	Generate $200$ linear systems $A_ix=b_i, i\in[200]$, where the positive definite matrices~$A_i$ are randomly generated with uniformly distributed eigenvalues in $(0,1)$ and the column vectors are set as $b_i := A_i x_i$, with  $x_i$ sampled from the standard normal distribution.
	\item
	For each $i\in{[200]}$, we generate 1000 perturbations $\{\Delta
b_j\}_{j\in[1000]}$ of norm	$10^{-6}$ and set $\Delta x_j :=
A_i^{-1} \Delta b_j$.  Then, for each ${j\in[1000]}$ we compute the
relative residual ratio in~\Cref{eq:relresratio}.  We then average over
$j$ to yield an estimate of the condition number, $\cond(A_i)$, of the
$i$-th system.

	\item
	We then check the resulting correlation between the
	following vectors (i) $(\kappa(A_i))_{i=1}^{200}$,
	(ii) $(\omega(A_i))_{i=1}^{200}$, and (iii) $(\sqrt{\omega(A_i^{-2})})_{i=1}^{200}$, with
	$(\cond(A_i))_{i=1}^{200}$, by comparing the corresponding linear regression models.
\end{enumerate}
\Cref{fig:erroranalysis_uniform}, page \pageref{fig:erroranalysis_uniform},
reveals a significant linear
correlation between $\cond$ and $\omega$, with a correlation
coefficient of $0.9251$ for $\omegatwo$ and
of $0.9062$ for $\omega$; whereas in contrast, $\cond$ and $\kappa$ are not linearly
correlated as the correlation coefficient is $0.4530$.

The same experiment with the eigenvalues of
the matrices $\{A_i\}_{i\in[200]}$ generated from the \emph{normal
standard distribution} are displayed in 
\Cref{fig:erroranalysis_normal}, page \pageref{fig:erroranalysis_normal}.
We get correlation coefficients: 
$0.7982 > 0.4847 > 0.0295$ with $\omegatwo,\omega,\kappa(A)$,
respectively, i.e.,~we cannot conclude existence of a  linear relation 
between $\cond$ and $\kappa,\omega$.


\label{page:figlinregrcondkappaomega}
\begin{figure}[ht!]\centering
\includegraphics[width = 0.5\textwidth]{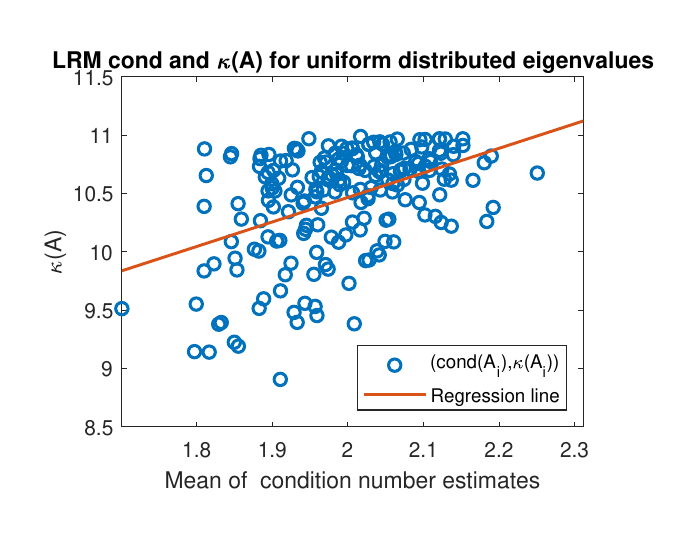}\hfill
\includegraphics[width = 0.5\textwidth]{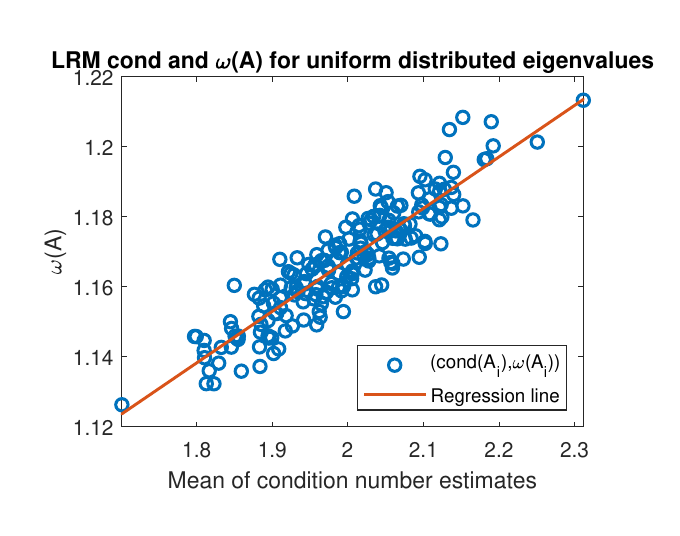}\hfill
\includegraphics[width = 0.5\textwidth]{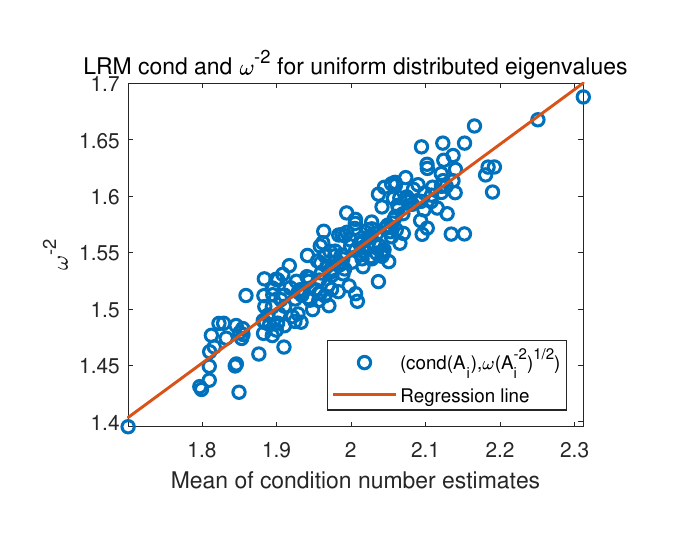}
\caption{Linear regression models (LRM) between $\cond$ and: 
$\kappa,\, \omega,\, \omegatwo$, respectively; uniformly distributed eigenvalues. 
}
\label{fig:erroranalysis_uniform}
\end{figure}
\begin{figure}[ht!]\centering
\includegraphics[width = 0.5\textwidth]{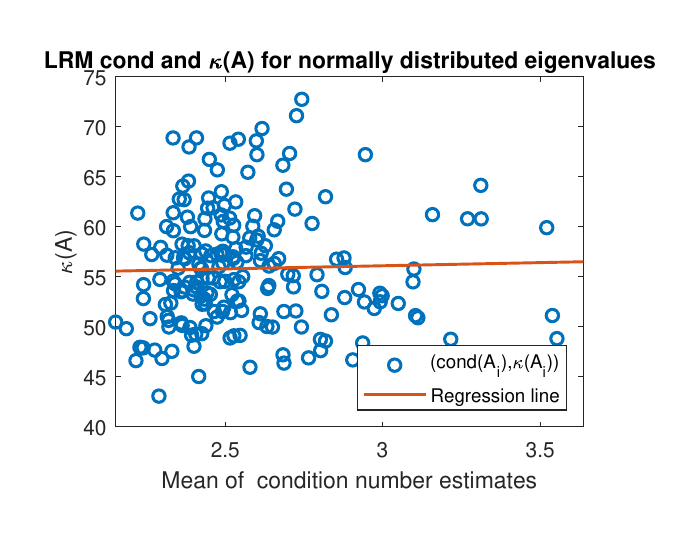}\hfill
\includegraphics[width = 0.5\textwidth]{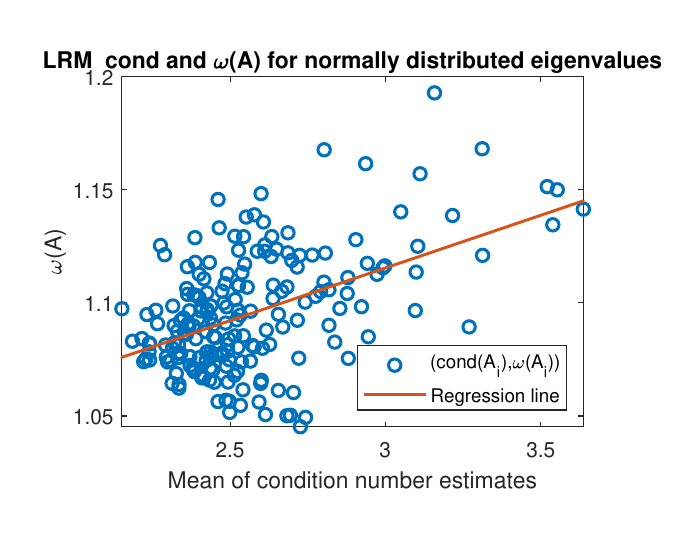}\hfill
\includegraphics[width = 0.5\textwidth]{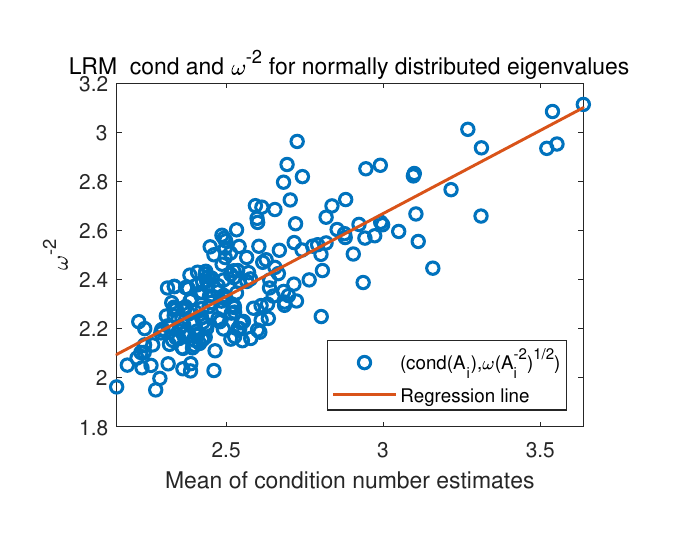}
\caption{Linear regression models (LRM) between $\cond$ and: 
$\kappa,\, \omega,\, \omegatwo$, respectively; normally distributed eigenvalues. 
}
\label{fig:erroranalysis_normal}
\end{figure}

\subsection{Eigenvalue Clustering}
\label{sect:itermethods}
As stated above, preconditioning is essential for iterative methods for
solving linear systems. And, many of the convergence analysis results
depend on clustering of eigenvalues, e.g.,~\cite{MR98j:65023}.
A typical comparison for the eigenvalues of $A\succ 0$ after
preconditioning with the optimal $\kappa,\omega,\textdef{$\omegatwo$}$ 
diagonal preconditioners is given in~\Cref{fig:smoothingeigs}
(the corresponding \textsc{Matlab} code is available online in \href{https://github.com/DavidTBelen/omega-condition-number.git}{https://github.com/DavidTBelen/omega-condition-number}).
\Cref{fig:smoothingeigs} clearly shows the improved clustering of
eigenvalues, as $\kappa$ essentially shifts the eigenvalues to reduce
the $\lambda_{\max}/\lambda_{\min}$ ratio, while both $\omega$ measures
move the eigevenvalues towards $1$ and promote clustering.
We see this in the large number of eigenvalues that are close to the mean
value for the $\omega$ optimal preconditioned matrices in the second
of the two figures in~\Cref{fig:smoothingeigs}.

The effect on
iterations for solving the system is given 
in~\Cref{section:NumericalTests}, below.
\begin{figure}[ht!]\centering
\vspace{-1.1in}
 \includegraphics[width = 0.5\textwidth]{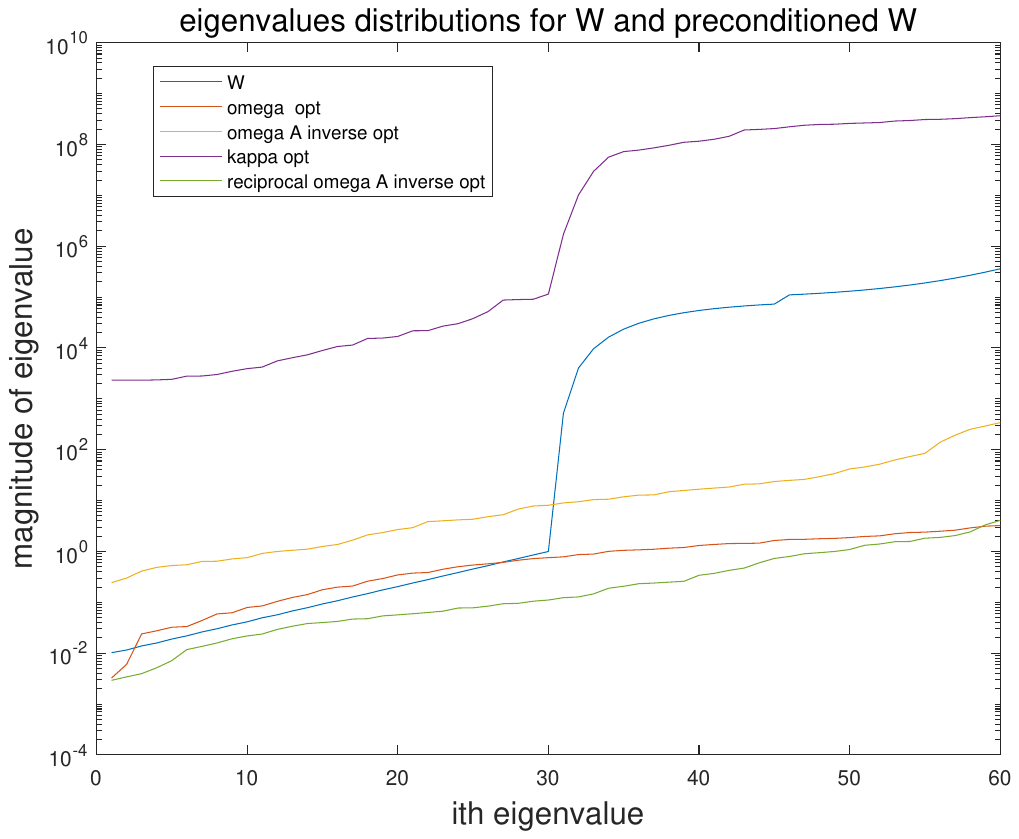}\hfill
 \includegraphics[width = 0.5\textwidth]{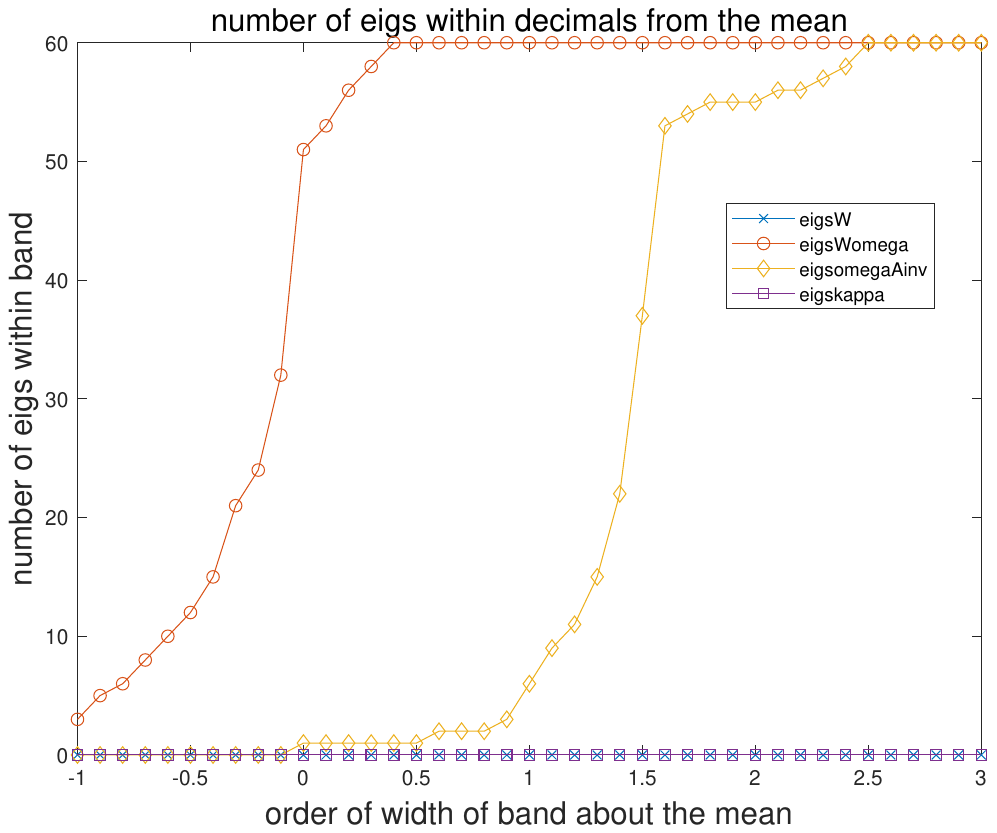}\hfill
\vspace{-1.1in}
 \caption{Comparison for clustering of eigenvalues pre-post
preconditioning}
\label{fig:smoothingeigs}
\end{figure}

\subsection{Incomplete Upper Triangular $\omega$-Optimal Preconditioner}
\label{sect:omegaoptpreconds}

	Approximations of the inverse of the Cholesky decomposition are widely
	used as preconditioners for linear systems. It is easy to 
	verify that the inverse of the Cholesky provides the optimal
	$\omega$ preconditioner. Indeed, let $W=R^TR$ be the Cholesky
	decomposition of $W$.  Then $\omega(R^{-T}WR^{-1}) = \omega(I) = 1$.
	However, it is well-known that sparsity can be lost when finding
	$R$ and $R^{-1}$. Therefore, permutation techniques are used when
	finding an incomplete Cholesky decomposition, e.g.,~\cite{MR3638573}.
	
	\index{$\Dt:\R^n \times \R^{t(k)}\to\R^{n\times n}$}
	In this section we see an interesting relationship between finding
	an $\omega$-optimal \emph{incomplete upper triangular preconditioner}
	and an incomplete Cholesky factorization, see~\Cref{thm:incompleteuppertri}.
	\label{page:pressparsity}
	Specifically, given an integer $2\leq
	k \leq n$, let
	$\alpha=(\alpha_{1,2},\alpha_{1,3},\alpha_{2,3},\ldots,\alpha_{1,k},\ldots,\alpha_{k-1,k})\in\R^{t(k-1)}$
	and $d\in\R^n$. We consider a   preconditioner in the form of 
	\begin{equation}\label{eq:Dt}
		\begin{aligned}
			\Dt (d,\alpha)& =  \Diag(d) +  \Trirk{(\alpha)} \\
			& = 
			\begin{pmatrix}
				d_1& \alpha_{1,2 }&  \alpha_{1,3} &  \ldots & \alpha_{1,k}&  0 & \ldots & 0 \cr
				0 &  d_2 &  \alpha_{2,3}& \ldots  & \alpha_{2,k} & 0 & \ldots &0 \cr
				0 & 0 & d_3 & \ddots  &\alpha_{3,k} &  0 & \ldots & 0 \cr
				0 & \ldots & \ldots & \ddots &\alpha_{k-1,k} &  0 & \ldots & 0  \cr
				\vdots & \ldots & \ldots & \ldots & d_{k} &  0 & \ldots & 0  \cr
				0 & \ldots  & \ldots  &  \ldots & 0 &  d_{k+1}& \ldots & 0 \cr
				0 & \ldots  & \ldots  &  \ldots & \ldots &  0 & \ddots & 0 \cr
				0 & \ldots  & \ldots  &  \ldots & \ldots &  0 & \ldots & d_n \cr
			\end{pmatrix},
		\end{aligned}
	\end{equation}
	where the linear mapping \textdef{$\Trirk:\R^{t(k-1)}\to \R^{n\times
			n}$} is defined accordingly. Its adjoint operator is  \textdef{$\Trirk^*
		= \trirk:\R^{n\times n} \to\R^{t(k-1)}$},
	$M\mapsto(M_{1,2},M_{1,3},M_{2,3},\ldots,M_{1,k},\ldots,M_{k-1,k})$.
	Moreover, the inverse maintains the same structure and sparsity pattern.
	
	Observe that if $k=n$ then $\Dt(d,\alpha)$ returns a complete upper
	triangular matrix. In that case it trivially follows that the
	$\omega$-optimal preconditioner will be given by the Cholesky
	decomposition. In any case, even when $k<n$, the
	$\omega$-optimal incomplete upper triangular preconditioner will
	be related to the Cholesky factorization. Therefore, we first recall
	the following recursive formula for computing the latter.
	
	\begin{remark}[Recursive formula for the Cholesky decomposition]
		Let $W\in\mathbb{S}^{k}$ be  a positive definite matrix and let $W=R^TR$
		be the Cholesky decomposition of $W$. We recall that the upper
		triangular Cholesky factor $R$ admits the following recursive
		columnwise construction for $j=1:k$:
		\begin{equation}\label{eq:chol-rec}
			\begin{aligned}
				R_{i,j} &=  \frac{1}{R_{i,i}} \left(W_{i,j} - \sum_{t=1}^{i-1} R_{t,j}
				\,  R_{t,i} \right), \quad \text{ for  } i = 1,\ldots, j-1, \\
				R_{j,j} &= \sqrt{ W_{j,j}  - \sum_{t=1} ^{j-1} R_{t,j}^2}.
			\end{aligned}
		\end{equation}
	\end{remark}
	
	\begin{theorem}
		\label{thm:incompleteuppertri}
		Let $W\in\mathbb{S}_{++}^n$, and let $W_{1:k,1:k}=R^TR$ be 
		the Cholesky decomposition, $k\in [n]$.
		The  $\omega$-optimal  incomplete upper
		triangular preconditioner in the form of \cref{eq:Dt} for $W$, i.e.,
		\begin{equation}
			\label{eq:omegaoptincchol}
			(\bd,\balpha) := \argmin_{(d,\alpha)\in\R_{++}^n\times \R^{t(k-1)}}
			\omega \left( \Dt(d,\alpha)^T W \Dt(d,\alpha) \right),
		\end{equation}
		is given by
		\begin{equation}\label{eq:bdbalpha} 
			\begin{aligned}
				\bd_j &= R_{j,j}^{-1},   &\text{for } j&\in[k];\\
				\bd_j & = W_{j,j}^{-1/2},  &\text{for } j&\in[k+1,n];\\
				\balpha_{i,j} & = - \frac{1}{R_{i,i}} \left( \sum_{s= i+1}^{j-1} R_{i,s}
				\balpha_{s,j} + R_{i,j} \bd_j\right),                      &\text{for }
				k&\geq j>i \geq 1.
			\end{aligned}
		\end{equation}
		We get, by abuse of notation with \cref{eq:bdbalpha},
		\[
		\Dt(\bar d,\bar \alpha)   = 
		\blkdiag \left(R^{-1},\Diag\left(\bar d_{[k+1,n]}\right)\right).
		\]
	\end{theorem}
	
		
\begin{proof}
	We divide the proof into three claims.
	
	\textbf{Claim 1:}  The $\omega$-optimal $\Dt$ preconditioner is obtained by $(\bd,\balpha)$ solving the nonlinear system
	\begin{equation}\label{eq:sequationt}
		\begin{bmatrix}
			\diag W \left(  \Diag(\bd) + \Trirk(\balpha) \right) \\
			\trirk W \left( \Diag(\bd) + \Trirk(\balpha) \right)
		\end{bmatrix}= \begin{pmatrix} \bd^{-1} \\ 0 \end{pmatrix},
	\end{equation}
	where $\bd^{-1} = (\bd_1^{-1}, \ldots, \bd_n^{-1})^T$.

	In order to prove this, and to ease the notation, fix $W$ and  consider the $\omega$-condition number, $f$ and $g$ as functions of a pair $(d,\alpha)\in\R_{++}^n \times \R^{t(k-1)}$. Namely, we set
	\[
	\omega_{+tk}(d,\alpha) = \frac{f_{+tk} (d,\alpha)}{g_{+tk} (d,\alpha)} := \frac{\tr\left( \Dt(d,\alpha)^T W \Dt(d,\alpha)\right)/n}{ \det \left( \Dt(d,\alpha)^T W \Dt (d,\alpha)  \right)^{1/n} }.
	\]
	Alternatively, we can rewrite $f_{+tk}$ as
	\begin{equation}\label{eq:ft}
		\begin{aligned}
			f_{+tk}(d,\alpha) &= \frac{1}{n} \tr\left( \Dt(d,\alpha)^T W \Dt(d,\alpha)\right) \\
			& = \frac{1}{n} \left\langle \Dt^* W \Dt (d,\alpha), \begin{pmatrix} d \\ \alpha\end{pmatrix} \right\rangle.
		\end{aligned}
	\end{equation}
	Hence,
	\begin{equation}\label{eq:gft}
		\begin{aligned}
			\nabla f_{+tk}(d,\alpha)  & = \frac{2}{n} \Dt^* W \Dt (d,\alpha)\\
			& = \frac{2}{n} \begin{bmatrix}
				\diag W \left(  \Diag(d) + \Trirk(\alpha) \right) \\
				\trirk W \left( \Diag(d) + \Trirk(\alpha) \right)
			\end{bmatrix}.
		\end{aligned}
	\end{equation}
	On the other hand, we have that
	\[
	g_{+tk}(d,\alpha) =  \det(W)\left( \prod_{i=1}^n d_i \right)^{\frac{2}{n}} \quad \text{ and } \quad \nabla g_{+tk}(d,\alpha) = \frac{2}{n} g_{+tk}(d,\alpha) \begin{pmatrix} d^{-1} \\ 0 \end{pmatrix},
	\]
	where $d^{-1} = (d_1^{-1}, \ldots, d_n^{-1})^T\in\R^n_{++}$.
	
	Therefore, the optimality condition for the pseudoconvex  function $\omega_{+t}$ is given  by
	\begin{equation}\label{eq:gomegat}
		\nabla \omega_{+tk} (d,\alpha) = K \left( \Dt^* W \Dt (d,\alpha) - f_{+tk}(d,\alpha)  \begin{pmatrix} d^{-1} \\ 0 \end{pmatrix} \right) =0,
	\end{equation}
	with $K:= 2 / (n \, g_{+tk}(d,\alpha) )>0$. Finally, observe that it suffices to obtain $(\bar{d},\bar{\alpha})\in\R^n_{++}\times\R^{t(k-1)}$ such that 
	\begin{equation}\label{eq:f-Claim1}
		\Dt^* W \Dt (\bar{d},\bar{\alpha}) -  \begin{pmatrix} \bar{d}^{-1} \\ 0 \end{pmatrix}  =0,
	\end{equation}
	as by~\cref{eq:ft} this immediately implies 
	\[
	f_{+tk}(\bar{d},\bar{\alpha}) =  \frac{1}{n}\left\langle   \begin{pmatrix} \bar{d}^{-1} \\ 0 \end{pmatrix}, \begin{pmatrix} \bar{d} \\ \bar{\alpha} \end{pmatrix} \right\rangle = 1,
	\]
	which in turn would yield \cref{eq:gomegat}. Thus, \cref{eq:f-Claim1} together with \cref{eq:gft} concludes this part of the proof.
	
	\textbf{Claim 2:} A solution  $(\bd,\balpha)$ to \cref{eq:sequationt} is
	given by $\bd_{i}= W_{i,i}^{-1/2}$,  for  $i\in[k+1,n]$, and with 
	\begin{equation}\label{eq:claim2}
		Q  := \Diag(\bd_{1:k}) + \Triu(\balpha) 
	\end{equation}
	being the inverse of the Cholesky decomposition of the matrix $W_{1:k,1:k}$.

	We start by fixing notation. Let $\widehat{W}: = W_{1:k,1:k}$ and   $\widetilde{W} := W_{k+1:n,k+1:n}$. Recall the definition of the operator $\Triu$ which applied to a vector  $\alpha =(\alpha_{1,2},\alpha_{1,3},\alpha_{2,3},\ldots,\alpha_{1,k},\ldots,\alpha_{k-1,k})\in\R^{t(k-1)}$ returns the  upper triangular matrix $\Triu(\alpha) = T\in \R^{k\times k}$ such that $T_{i,j}=\alpha_{i,j}$ if $1\leq i< j \leq n$, and $T_{i,j} = 0$ otherwise. The adjoint of $\Triu$ is denoted as $\triu$.  Then the system \cref{eq:sequationt} can be split into the  two equations
	\begin{equation}\label{eq:sequationt-1}
		\diag W \left(  \Diag(\bd) + \Trirk(\balpha) \right)  = 
		\begin{bmatrix}
			\diag\widehat{W}\left(\Diag(\bd_{1:k}) + \Triu(\balpha)\right) \\
			\diag\widetilde{W}\Diag(\bd_{k+1:n})
		\end{bmatrix}
		= \bar{d}^{-1}
	\end{equation}
	and
	\begin{equation}\label{eq:sequationt-2}
		\trirk W \left( \Diag(\bd) + \Trirk(\balpha) \right)   =
		\triu \widehat{W} \left(\Diag(\bd_{1:k}) + \Triu(\balpha)\right) 
		=0.
	\end{equation}
	Observe that the variables $\bd_{k+1},\ldots,\bd_{n}$ only appear in the
	lower block of~\cref{eq:sequationt-1}, that can be directly solved to
	obtain $\bd_{i} = W_{i,i}^{-1/2}$ for all $i\in[k+1,n]$.
	
	On the other hand, the variables $\bd_1,\ldots,\bd_k$ and $\balpha$  are present in \Cref{eq:sequationt-2} and the upper block of \Cref{eq:sequationt-1}. Nonetheless, by taking into account that if $n=k$  then $\Triu = \Trirk$, it is easy to check that these equations define the $\omega$-optimal triangular preconditioner of the matrix $\widehat W\in\mathbb{S}^k_{++}$. Therefore we conclude that $Q$ coincides with the inverse of the Cholesky factorization of $\widehat W$. 
	
	\textbf{Claim 3:}  Let $Q  := \Diag(\bd_{1:k}) + \Triu(\balpha) $ be the inverse of the Cholesky decomposition of $\widehat{W}$. Then 
	$(\bd_{1:k},\balpha)$ is given as in \cref{eq:bdbalpha}.

	Let $\widehat{W}=R^T R$ be the Cholesky decomposition of $\widehat{W}$, where
	\[
	R = \begin{pmatrix}
		R_{1,1} & R_{1,2} &  R_{1,3} &  \ldots & R_{1,k} \cr
		0 &  R_{2,2} &  R_{2,3} & \ldots & R_{2,k} \cr
		\vdots & \ldots & \ldots & \ldots & \vdots \cr
		0 & \ldots & \ldots & 0 & R_{k,k}
	\end{pmatrix},
	\]
	and the entries are given as in~\cref{eq:chol-rec}. Let $Q=R^{-1}$ be the matrix defined in \cref{eq:claim2}. We now use the equation $RQ=\Id$ to obtain an expression of $Q$ in terms of $R$. We have:
	{\footnotesize
		\[
		\Id = 
		\begin{pmatrix}
			R_{1,1} & R_{1,2} &  R_{1,3} &  \ldots & R_{1,k} \cr
			0 &  R_{2,2} &  R_{2,3} & \ldots & R_{2,k} \cr
			\vdots & \ldots & \ddots & \ldots & \vdots \cr
			\vdots & \ldots & \ldots  & R_{k-1,k-1} & R_{k-1,k} \cr
			0 & \ldots & \ldots & 0 &  R_{k,k}
		\end{pmatrix} 
		\begin{pmatrix}
			\bd_1& \balpha_{1,2 }&  \balpha_{1,3} &  \ldots & \balpha_{1,k-1}& \balpha_{1,k} \cr
			0 &  \bd_2 &  \balpha_{2,3}& \ldots & \ldots & \balpha_{2,k} \cr
			\vdots & \ldots & \ldots & \ldots & \ldots &\vdots \cr
			\vdots & \ldots & \ldots & \ldots & \bd_{k-1} & \balpha_{k-1,k} \cr
			0 & \ldots & \ldots & \ldots  &  0 & \bd_k
		\end{pmatrix}.
		\]
	}%
	For each column $j \in{ [k]}$ of $Q$, this leads to the following linear system of  $j$ equations:
	\begin{subequations}
		\begin{align}
			1 & = R_{j,j} \, \bd_{j}, \label{eq:QR-syst-1} \\
			0 & = R_{j-1,j-1} \, \balpha_{j-1,j} + R_{j-1,j} \, \bd_{j} ,  \\
			& \vdots \notag \\
			0 & = R_{j-\ell+1,j-\ell+1}\, \balpha_{j-\ell+1,j} +  \sum_{s=j-\ell+2}^{j-1} R_{j-\ell+1,s} \, \balpha_{s,j} + R_{j-\ell+1,j}\, \bd_{j}, \label{eq:QR-syst-j} \\
			&\vdots \notag \\
			0 & = R_{1,1} \, \balpha_{1,j}  +  \sum_{s=2}^{j-1} R_{1, s}\,  \balpha_{s,j} + R_{1,j} \, \bd_{j} \label{eq:QR-syst-last}.
		\end{align}
	\end{subequations}
	
	Equation \cref{eq:QR-syst-1} readily implies that $\bd_{j} = R_{j,j}^{-1}$ for all $j \in{[k]}$.
	Moreover, for any $\ell\in{[2,j]}$, we can  solve  \cref{eq:QR-syst-j} for getting an expression for $\balpha_{j-\ell+1,j}$ in terms of $\bd_{j}, \balpha_{j-1,j}, \ldots, \balpha_{j-\ell+2,j}$. This yields
	\begin{equation}\label{eq:alpha_ell}
		\balpha_{j-\ell+1,j} = -\frac{1}{R_{j-\ell+1,j-\ell+1}} \left(  \sum_{s=j-\ell+2}^{j-1} R_{j-\ell+1,s} \, \balpha_{s,j} + R_{j-\ell+1,j}\, \bd_{j} \right),
	\end{equation}
	which concludes  Claim 3 and the proof.
\end{proof}

We conclude this section with a simple \textsc{Matlab}'s code for an efficient computation of the $\omega$-optimal incomplete upper triangular preconditioner.

\begin{lstlisting}[style=Matlab-editor]
	%%% Function for computing the $\omega$-optimal incomplete upper triangular preconditioner
	% Input:
	%       - W <- pos. def. matrix
	%       - k <- size of the triangular block
	% Output:
	%       - D <-  optimal preconditioner minimizing omega(D'*W*D)
	function D = i_upper_tri_preconditioner(W,k)
	n = length(W);
	tempR = W(1:k,1:k);
	R = chol(tempR);
	tempW = W(k+1:n,k+1:n);
	tempD = diag(diag(tempW).^(-1/2));
	D = blkdiag(inv(R),tempD);
	end
\end{lstlisting}

\section{Optimal Conditioning for Generalized Jacobians}\label{sec:genJac}
We now consider the problem of improving conditioning
for  low rank updates of very ill-conditioned (close to singular) positive definite matrices. 
\subsection{Preliminaries}
More precisely, given a positive definite matrix $A\in\Snpp$  and a
matrix $U\in\R^{n\times t}$ with $t<<n$,
we aim to find $\gamma\in\R^t$ so as to minimize the condition number of the low rank update
\begin{equation}\label{eq:lowrankupdate}
A + U\Diag(\gamma)U^T.
\end{equation}
This kind of updating arises when finding  generalized Jacobians in
nonsmooth optimization. We provide insight on the 
problem in the following~\Cref{ex:generalizedJacobians}.
\begin{example}[Generalized Jacobians]
\label{ex:generalizedJacobians}
In many nonsmooth and semismooth Newton methods one aims to find a root of a function $F:\R^n\to\R^n$ of the form
\[
F(y):= B(v+B^T y)_+-c,
\]
where $B\in\R^{n \times m}$, $v\in\R^m$, $c\in\R^n$ and $(\,\cdot\,)_+$
denotes the projection onto the nonnegative orthant,
e.g.,~\cite{CensorMoursiWeamsWolk:22,HuImLiWo:21,qi2006quadratically}. 
At every iteration of these algorithms a generalized Jacobian of $F$ 
of the form
\[
J:=\sum_{i\in\cI_+} B_{i}B_i^T + \sum_{j\in\cI_0} \gamma_j B_j B_j^T,
	\text{ with  }\gamma_j\in{[0,1]},
\]
is computed.
Here $B_i$ and $B_j$ denote columns of $B$ over the set of indices
\[
{\small
\begin{array}{rcl}
	\cI_+&:=&\{ i\in [m] : \, (v+B^T y)_i >0\};\text{ and} \cr
	\cI_0 &:=& \{ j\in [m] : (v+B^Ty)_j = 0 \text{ and } (B_j)_{j\in\cI_0} \text{ is a maximal linearly independent set}\}.
\end{array}
}
\]
The generalized Jacobian $J$, is usually singular. It is used to
obtain a Newton direction $d\in\R^n$ by solving  a least-square problem
for the system $(J+\epsilon I) \,d = -F( y)$, where  $\epsilon I$, with
$\epsilon >0$,  is analogous to the regularization term of the
well-known \emph{Levenberg--Marquardt method}. Thus, this linear system
is in general very ill-conditioned. This makes preconditioning by optimal
updating appropriate.

The optimal preconditioned update can be done in our framework as we 
start with
\begin{equation}\label{eq:GJ-setting}
A := \sum_{i\in\cI_+} B_{i}B_i^T + \epsilon I, \quad U =
[B_j]_{j\in\cI_0},
\end{equation}
and then find an optimal low rank update as in \cref{eq:lowrankupdate};
done with additional  box constraints on  $\gamma$, namely, 
$\gamma\in{[0,1]^t}$.
\end{example}

Similar conditioning questions also appear in the normal
equations matrix, $ADA^T$, in interior point methods, e.g.,~modifying the
weights in $D$ appropriately to avoid 
ill-conditioning~\cite{MR4594481,MR4514091}.
For other related work on minimizing condition numbers for low rank updates see,
e.g.,~\cite{MR2765492,MR2288015}.

Here, we propose obtaining an optimal conditioning of the update~\cref{eq:lowrankupdate} by
using  the $\omega$-condition number of~\cite{DeWo:90}, instead of the
classic $\kappa$-condition number.
The $\omega$-condition number presents some advantages with respect to
the classic condition number, since it is differentiable and pseudoconvex in the interior of the positive semidefinite cone, which facilitates addressing minimization problems involving it.
Our empirical results show a significant decrease in the number of
iterations required for a requested accuracy in the residual.

\subsection{Optimal Conditioning for Rank One Updates}
\label{sect:optrankone}
We first consider the special case where the update is rank one.
Related eigenvalue results for rank one updates are well known in the
quasi-Newton literature, e.g.,~\cite{DennSch:79,Schna:77}.
We include this special rank one case as it yields 
 interesting  results.
The general rank-$t$ update is studied in~\Cref{sect:lowrankoptG}, below.
\begin{theorem}
\label{thm:optrankonegamma}
Suppose we have a given $A\in \Snpp$ and   $u\in \Rn$.
Let $A = QDQ^T$ be the (orthogonal) spectral 
decomposition of $A$.
Let $U = uu^T$ and define the rank one update
\[
A(\gamma) = A+\gamma U, \, \gamma \in \R.
\]
 Set
\begin{equation}
\label{eq:wbarDeps}
w =  D^{-1/2}Q^Tu, 
\end{equation}
and
\begin{equation}\label{eq:gammas}
\gamma^* = \frac{ \trace(A) \|w\|^2 - n\|u\|^2 } 
{(n-1)\|u\|^2\| w\|^2}.
\end{equation}
Then, 
$
\gamma^*\in{\left]-\|w\|^{-2},+\infty\right[}
$
provides the  $\omega$-optimal conditioning,
i.e.,
\begin{equation}\label{eq:rank1min}
\gamma^* = \argmin_{A(\gamma)\succ 0} \omega(\gamma).
\end{equation}
\end{theorem}

\begin{proof}
Let
\[
f(\gamma) := \trace(A(\gamma))/n \quad \text{ and  } \quad 
g(\gamma) := \det(A(\gamma))^{1/n}.
\]
We want to find the optimal $\gamma$ to minimize the condition number
\[
\omega(\gamma) = f(\gamma)/g(\gamma)
\]
subject to $A(\gamma)$ being positive definite. By \Cref{prop:precond,item:globminomega}, $\omega:\mathbb{R}\to\R;\,\gamma \to\omega(\gamma)$ is pseudoconvex as long as $A(\gamma)\succ 0$. We prove that the later is true for $\gamma$ belonging to an open interval in the real line. Indeed, let $A= Q D Q^T$ be the spectral decomposition of $A$ and define 
\begin{equation}\label{eq:Wspect}
 w =  D^{-1/2}Q^Tu\quad \text{ and } \quad
 W =  w w^T = D^{-1/2} Q^Tuu^TQD^{-1/2}.
\end{equation}
Then we can rewrite 
\begin{equation}\label{eq:Agamma}
A(\gamma) = QD^{1/2}(I +\gamma  W)D^{1/2}Q^T,
\end{equation}
which is positive definite if and only if the rank one update of $I$,  $I +\gamma  W $, belongs to the cone of positive definite matrices.  Now, note that the eigenvalues of this term are $\lambda_1=1$, with multiplicity $n-1$, and  $\lambda_2 = 1 +\gamma \|w\|^2$ with multiplicity $1$. We then conclude that
\[
A(\gamma) \in \Snpp \; \Longleftrightarrow \; \gamma \in{\left]- \frac{1}{\|w\|^2}, +\infty \right[},
\]
in which case $\lambda_2 >0$. 
Moreover, $\omega(\gamma)$ tends to $\infty$ as $\gamma$ approaches the
extreme of the above interval. Therefore $\omega$ possesses a  minimizer
in  the open interval, $\gamma^*\in{\left]- \|w\|^{-2}, +\infty
\right[}$, that satisfies $\omega^{\prime}(\gamma^*) = 0$. Note that since $\omega$ is pseudoconvex the fact that its derivative is equal to zero is also a sufficient condition for global optimality (see \Cref{thrm:pseudocnvprog} below).

\label{page:cholspectral}
In the following we obtain an explicit expression for the (unique) minimizer of
\Cref{eq:rank1min}, $\gamma^*$,  by studying
the zeros of $\omega^{\prime}$.  
Using the notation introduced in~\cref{eq:Wspect}, $f$ and its
derivative are expressed as           
\[
f(\gamma) = \left(\trace(A)  +\gamma \|u\|^2 \right)/n \quad  \text{ and } \quad
f^\prime(\gamma) = \|u\|^2/n,
\]
respectively. By making use of~\cref{eq:Agamma}, $g$ becomes
\[
g(\gamma) := 
\left(\det( A)
\det( I + \gamma W )\right)^{1/n},
\]
since $\det(D) = \det(A)$. As explained above the eigenvalues of
$ I + \gamma  W$, are $\lambda_1 = 1 +\gamma \| w\|^2$, and the
others are all $1$, which yields that
\[
g(\gamma) =\left( \det(A) (1+ \gamma \| w\|^2) \right)^{1/n} 
          = \det(A)^{1/n}  (1+ \gamma \|w\|^2)^{1/n}.
\]
We get
\[
g^\prime(\gamma) =  
\frac 1n
           \det(A)^{1/n}  \| w\|^2
           (1+ \gamma \| w\|^2)^{(1-n)/n}.
\]
The derivative of $\omega$ is then obtained as follows
\label{page:typosq}
\begin{equation}
\label{eq:omegaprime}
\begin{aligned}
\omega^\prime (\gamma)
&= \frac{f^\prime(\gamma) g(\gamma)- f(\gamma) g^\prime(\gamma)}{g(\gamma)^2} \\
& =
\frac{1}{g(\gamma)^2}\left[ \frac{\|u\|^2}{n} \det(A)^{1/n} \left(1 +\gamma\|w\|^2\right)^{1/n} \right. \\
&  \quad  \left.  - \frac{\|w\|^2}{n^2} \left(\trace(A) + \gamma \|u\|^2 \right) \det(A)^{1/n} \left( 1+ \gamma\|w\|^2\right)^{(1-n)/n}\right] \\
& = \frac{\det(A)^{1/n}}{g(\gamma)^2 n^2} \left( 1+\gamma\|w\|^2\right)^{(1-n)/n} \left[n \|u\|^2  +(n-1) \gamma \|u\|^2\|w\|^2 - \trace(A) \|w\|^2 \right].
\end{aligned}
\end{equation}
A simple computation shows that this derivative is $0$ only when $\gamma$ attains the value
\begin{equation}\label{eq:omegaprime0}
\gamma^* =
          \frac{ \trace(A) \|w\|^2 -
                            n\|u\|^2 }
             {(n-1)\|u\|^2\|  w\|^2},
\end{equation}
that has to be in the open interval $\left]-\|w\|^{-2},+\infty\right[$. Since $\omega$ is pseudoconvex, we conclude that $\gamma^*$ is the $\omega$-optimal conditioning that solves~\Cref{eq:rank1min}. 
\end{proof}

Equivalently, we can deduce an expression for the  $\omega$-optimal conditioning by making use of the Cholesky decomposition of $A$ instead of the spectral decomposition. This is gathered in our next corollary.  The proof follows from the same calculations than~\Cref{thm:optrankonegamma} and thus is omitted.
\begin{corollary}
Given $A$ and $U$ as in~\Cref{thm:optrankonegamma}. Let $A=LL^T$ be the Cholesky decomposition of $A$. Then, the formula for the  $\omega$-optimal conditioning $\gamma^*$ in~\Cref{eq:gammas} holds with the replacement 
\[
w \leftarrow L^{-1}u.
\]
\end{corollary}

\label{page:oldcorlem}

As shown in \Cref{ex:generalizedJacobians}, in some applications
the preconditioner multiplier $\gamma$ is required 
to take values in the interval $[0,1]$.
In the following, we analize the optimal $\omega$-preconditioner
for the rank $1$ update subject to this interval constraint.

\begin{corollary}\label{cor:rankoneconst}
Let the assumptions of~\Cref{thm:optrankonegamma} hold and let $\bar\gamma$  
be the  $\omega$-optimal conditioning in the interval $[0,1]$, i.e., 
\[
\bar \gamma = \arg \min_{\stackrel{0\leq \gamma\leq 1}{A(\gamma)\succ 0}}
        \omega(\gamma). 
\]
Then, if $\gamma^*\in{]-\|w\|^2,+\infty[}$ is the $\omega$-optimal ``unconstrained'' conditioning obtained in~\Cref{thm:optrankonegamma}, the following hold:
\begin{enumerate}[(i)]
\item\label{it:01-0}
If $\gamma^*\in{[0,1]} \implies \bar \gamma = \gamma^*$;
\item\label{it:01-1}
If $\gamma^*< 0 \implies \bar \gamma = 0$;
\item\label{it:01-2}
If $\gamma^* > 1 \implies \bar \gamma = 1$.
\end{enumerate}
\end{corollary}
\begin{proof}
\Cref{it:01-0} In this case, since   $\gamma^*$ is the global optimum of $\omega$ in $]-\|w\|^2,+\infty[$, it would also be so in the interval $[0,1]$.

For~\Cref{it:01-1,it:01-2}, it suffices to observe that, by~\Cref{eq:omegaprime,eq:omegaprime0}, when $\gamma^* <0$ (respectively, $\gamma^* > 1$) the derivative of $\omega$ is monotonically increasing (respectively, decreasing) in the interval $[0,1]$.
\end{proof}

\subsection{Optimal Conditioning with a Low Rank Update}
\label{sect:lowrankoptG}
\label{page:sizeZ}
We now consider the case where the update is low rank. 
We need the following notations.
For a matrix $Z\in \R^{n\times t}$, we use \textsc{Matlab} notation and define
the function {\color{black} \textdef{$\norms(Z) : \R^{n\times t}\to \R^t$}} as the
(column) vector of column $2$-norms of $Z$. We let \textdef{$\normsa(Z)$}
denote the vector of column norms with each norm to the power $\alpha$.

\index{$u\circ w$, Hadamard product}
\index{Hadamard product, $u\circ w$}
\label{page:simplifypf}
\begin{theorem}[Rank $t$-update]
\label{thm:lowrankupdate}
Let $A\in \Snpp, \, U=[u_1,\ldots,u_t]\in \R^{n\times t}$, be given with
$ n>t\geq 2$, and $\norms(U)>0$. Set
\[
A(\gamma) = A + U \Diag(\gamma) U^T,
             \text{  for  }   \gamma \in \R^t.
\]
Let the spectral decomposition of $A$ be given by $A=QDQ^T$, define $ w_i  =  D^{-1/2}Q^Tu_i, i\in[t]$, as in
\cref{eq:wbarDeps}, with $ W = \begin{bmatrix}w_1\ldots 
w_t\end{bmatrix}$. Let
\begin{equation}\label{eq:kUbU}
\begin{array}{rcl}
K(U)  
&=&
\begin{bmatrix}n\Diag\left(\normst(U)\right) -
e\normst(U)^T
\end{bmatrix},
 \\ b(U) 
&=& 
\begin{pmatrix}\trace(A)e- 
n\Diag\left(\normst(W)\right)^{-1}\normst(U)
\end{pmatrix},
\end{array}
\end{equation}
\label{page:dotslash}
where $e$ denotes the vector of all ones. Then, the  $\omega$-optimal conditioning,
\begin{equation}\label{eq:minlowrank}
\gamma^* = \argmin_{A(\gamma)\succ 0} \omega(\gamma),
\end{equation}
is given component-wise for $i\in[t]$ by
\begin{equation}
\label{eq:optgammalowrank}
\begin{aligned}
(\gamma^*)_i  
&=
(K(U)^{-1}b(U))_i
\\&= \frac{1}{(n-t)\|u_i\|^2}\left(\tr(A) - (n-t)\frac{\|u_i\|^2}{\|w_i\|^2} - \sum_{j=1}^{t}\frac{\|u_j\|^2}{\|w_j\|^2}\right).
\end{aligned}
\end{equation}
\end{theorem}
\begin{proof}
Let $A\succ 0$ and  
\[
U = \begin{bmatrix} u_1&\ldots&u_t\end{bmatrix}\in\R^{n\times t}, \, \quad \text{ with }
n>t\geq 2.
\]
We consider the update of the form
\[
A(\gamma) = A+ U \Diag(\gamma) U^T = A +
\sum_{i=1}^t  \gamma_i u_i u_i^t,\quad  \gamma\in \R^t.
\]
Same than in~\Cref{thm:optrankonegamma}, we start characterizing an open subset of $\R^t$ where $A(\gamma)$ is positive definite. In order to do this, we 
again transform the problem using the spectral decomposition of $A$, $A=QDQ^T$, 
and setting 
\[
 w_i = D^{-1/2} Q^T u_i  \quad \text{ and }  \quad  W_i = w_i w_i^T \quad \text{ for } i\in[t].
\]
Then, we can express  $A(\gamma)$ as
\[
\begin{array}{rcl}
A(\gamma) &=&A+ U\Diag(\gamma) U^T \\
&=&
QD^{1/2}\left(I + D^{-1/2} Q^T U \Diag(\gamma) U^T Q D^{-1/2}\right)D^{1/2} Q^T  \\
&=&
QD^{1/2} \left( I + \displaystyle\sum_{i=1}^t \gamma_i (D^{-1/2} Q^T u_i) (u_i^T Q D^{-1/2} ) \right) D^{1/2} Q^T \\
& =& 
QD^{1/2}  \left( I + \displaystyle\sum_{i=1}^t \gamma_i W_i \right)D^{1/2} Q^T.
\end{array}
\]
By repeatedly making use of the formula for the determinant of the 
sum of an invertible matrix and a rank one matrix (see,
e.g.,~\cite[Example~4]{miller1981inverse}),
we obtain the following expression for the determinant of $A(\gamma)$
\begin{equation}\label{eq:detAgammalw}
\det\big( A(\gamma)\big) = \det(A) \prod_{i=1}^t (1+\gamma_i\| w_i\|^2 ).
\end{equation}
Consequently, $A(\gamma)$ is nonsingular and, by continuity of the eigenvalues, positive definite for $\gamma$ belonging to the set
 \begin{equation}\label{eq:Omegalr}
\Omega := \left]-\frac 1{\|w_1\|^2},+\infty\right[ \;\times\; \left]-\frac 1{\|w_2\|^2},+\infty\right[ \times\;  \ldots\; \times \;\left]-\frac 1{\|w_t\|^2},+\infty\right[.
 \end{equation}
 Now, note that the constraint $A(\gamma) \succ 0$ is a positive definite constraint, so it is convex. Therefore, if there exists some $\gamma$ outside of $\Omega$ such that $A(\gamma) \succ 0$, we would lose the convexity of the feasible set, since $A(\gamma)$ is singular on the boundary of $\Omega$. This implies that 
 \[
 A(\gamma) \succ 0 \Longleftrightarrow \gamma \in \Omega.
 \]

Moreover, since $\omega(\gamma)\to+\infty$ as $\gamma$ tends to the border of $\Omega$ or to $+\infty$, we can ensure that $\gamma$ has a minimizer in $\Omega$. Since the function is pseudoconvex on this open set, the global minimum is attained at a point $\gamma^*$ such that $\nabla\omega(\gamma^*) = 0$. Next, we prove that $\gamma^*$ is given by~\Cref{eq:optgammalowrank}.

For this, note that $f(\gamma)$ can be expressed as
\[
f(\gamma) =  \frac 1n \trace(A+ U\Diag(\gamma)U^T)= 
\frac 1n \left(\trace (A) + \sum_{i=1}^t
\gamma_i\|u_i\|^2\right)=
\frac 1n \left(\trace (A) +
\gamma^T\normst(U)\right),
\]
and its gradient is $\nabla f(\gamma)= \frac 1n
\normst(U)$.  On the other hand, by~\Cref{eq:detAgammalw} $g(\gamma)$ can be expressed as 
\[
g(\gamma) = \det(A)^{1/n} \left( \prod_{i=1}^t (1+\gamma_i\| w_i\|^2 )\right)^{1/n}.
\]
The gradient of $g(\gamma)$ is then given component-wise by
\[
\begin{array}{rcll}
	\dfrac {\partial g(\gamma)}{\partial \gamma_j}
	&=& \dfrac{1}{n}\det(A)^{1/n} \left( \displaystyle\prod_{i=1}^t (1+\gamma_i\| w_i\|^2)\right)^{(1-n)/n} \left( \displaystyle\prod_{i=1, i\neq j}^t (1+\gamma_i\| w_i\|^2) \right)\| w_j\|^2 \\
	&=& \dfrac{1}{n}\det(A)^{1/n} \left( \displaystyle\prod_{i=1}^t (1+\gamma_i\| w_i\|^2)\right)^{1/n} (1+\gamma_j\|w_j\|^2)^{-1}\| w_j\|^2\\
	&=&\dfrac{g(\gamma)}{n}\dfrac{\|w_j\|^2}{1+\gamma_j\|w_j\|^2},&j\in[t].
\end{array}
\]
We make use of these expressions in order to compute the partial derivatives of $\omega$. For every $j\in[t]$, we have
\begin{equation}
\label{eq:partomega}
\begin{aligned}
\frac{\partial \omega(\gamma)}{\partial \gamma_j} & = \frac{1}{g(\gamma)^2}\left[\frac{\partial f (\gamma)}{\partial \gamma_j} g(\gamma) - f(\gamma) \frac{\partial g(\gamma) }{\partial \gamma_j}\right]\\
&= \frac{1}{n^2g(\gamma)}\left[n\|u_j\|^2 - \frac{\|w_j\|^2\big(\tr(A) + \gamma^T\normst(U)\big)}{1+\gamma_j\|w_j\|^2}\right].
\end{aligned}
\end{equation}
\label{page:typosfix}
Since $n^2g(\gamma)>0$, the $j$-th partial derivative of $\omega$ is zero if, and only if,
\[
n\|u_j\|^2 - \frac{\|w_j\|^2\big(\tr(A) + \gamma^T\normst(U)\big)}{1+\gamma_j\|w_j\|^2}=0.
\]
Therefore,
%
the minimum of the  pseudoconvex function is obtained as the solution of the linear system  defined by the $t$ equations  

\[
(n-1) \|u_k\|^2 \gamma_k - \sum_{i=1, i\neq k}^t\|u_i\|^2 \gamma_i= \tr(A) - n \frac{\|u_k\|^2}{\|  w_k\|^2}, \quad  k\in[t].
\]
Equivalently,
{\footnotesize
\[
\begin{bmatrix}
(n-1)\|u_1\|^2 & -\|u_2\|^2 &\ldots & & -\|u_t\|^2  \cr
-\|u_1\|^2 & (n-1)\|u_2\|^2 & -\|u_3\|^2 &\ldots &-\|u_t\|^2  \cr
\ldots \cr
-\|u_1\|^2 &  \ldots &\ldots &-\|u_{t-1}\|^2 &(n-1)\|u_t\|^2  \cr
\end{bmatrix}\gamma = 
\begin{pmatrix}
\trace(A) - n \|u_1\|^2/\| w_1\|^2\cr
\ldots \cr
\trace(A) - n \|u_t\|^2/\| w_t\|^2
\end{pmatrix}.
\]
}
This is further equivalent to 
\[
\begin{bmatrix}n\Diag(\normst(U)) - e
\normst(U)^T
\end{bmatrix}\gamma= 
\begin{pmatrix}\trace(A)e- 
n\Diag\left(\normst(W)\right)^{-1}\normst(U)
\end{pmatrix},
\]
which is  the system $K(U) \gamma = b(U)$ using the notation in~\Cref{eq:kUbU}.

Now we derive an explicit expression for the optimal $\gamma$. In order to do this, note that $K(U)$ is given as the sum   of an invertible matrix, $n\Diag(\normst(U))$, and an outer product of vectors, $-e\normst(U)^T$. By the \textdef{Sherman-Morrison formula}, this sum is invertible if and only if 
\[
1-\frac 1n \normst(U)^T \Diag(\normst(U))^{-1} e  \neq 0.
\]
This is always true for $t <n$. Indeed, we have
\[
1-\frac 1n \normst(U)^T \Diag(\normst(U))^{-1} e = 1- \frac 1n e^T e = 1 -\frac tn >0.
\]
Moreover, we obtain the following expression for the inverse
{\small
\[
\begin{aligned}
\left(n \right. &\left.\Diag(\normst(U)) -e\normst(U)^T\right)^{-1}  \\
&= \frac 1n \Diag(\normst(U))^{-1}  + \frac 1{\left( 1-\frac tn \right)n^2} \Diag(\normst(U))^{-1} e \normst(U)^T \Diag(\normst(U))^{-1}  \\
& =
\frac 1n \Diag(\normsti(U)) + \frac  1{(n-t)n}  \Diag(\normsti(U)) e e^T. \\
\end{aligned}
\]
}
Therefore, the inverse  of $K(U)$ in matrix form is given by
\[
K(U)^{-1} =
\frac1n
\begin{bmatrix}
\frac{1}{\|u_1\|^2} & 0 & \ldots & 0 \cr
0 & \frac{1}{\|u_2\|^2}  & \ldots & 0 \cr
\vdots &  \vdots & \ddots & \vdots \cr
0 & 0 & \ldots & \frac{1}{\|u_t\|^2} \cr
\end{bmatrix}
+
\frac 1{(n-t)n}
\begin{bmatrix}
\frac{1}{\|u_1\|^2} &\frac{1}{\|u_1\|^2} &\ldots & \frac{1}{\|u_1\|^2}  \cr
\frac{1}{\|u_2\|^2} &\frac{1}{\|u_2\|^2} &\ldots & \frac{1}{\|u_2\|^2}  \cr
\vdots &  \vdots & \vdots & \vdots \cr
\frac{1}{\|u_t\|^2} &\frac{1}{\|u_t\|^2} &\ldots & \frac{1}{\|u_t\|^2}  \cr
\end{bmatrix}.
\]
Finally, we obtain $\gamma^*$ by calculating the product $\gamma^* = K(U)^{-1}b(U)$ which yields
\begin{equation}\label{eq:optlr}
\gamma_i^* = \frac{1}{(n-t)\|u_i\|^2}\left(\tr(A) - (n-t)\frac{\|u_i\|^2}{\|w_i\|^2} - \sum_{j=1}^{t}\frac{\|u_j\|^2}{\|w_j\|^2}\right),
\end{equation}
for all $i\in[t]$. Since $\gamma^*$ is the unique zero of the gradient of $\omega$, we conclude that it belongs to $\Omega$ and solves~\Cref{eq:minlowrank}.
\end{proof}
We note that the  $\omega$-optimal conditioning for the rank one update in~\Cref{thm:optrankonegamma} is obtained from~\Cref{eq:optgammalowrank} when $t=1$. On the other hand, we can also employ the Cholesky decomposition of $A$ to derive the  $\omega$-optimal conditioning in~\Cref{thm:lowrankupdate}. We state this in the following corollary.

\begin{corollary}\label{cor:lrChol}
Given $A$ and $U$ as in~\Cref{thm:lowrankupdate}. Let $A=LL^T$ be the Cholesky decomposition of $A$. Then, the  formula for the  $\omega$-optimal conditioning $\gamma^*$ in~\Cref{eq:optgammalowrank} holds
with  the replacement 
\[
 w_i \leftarrow L^{-1}u_i,\, \quad   i\in[t].
\]
\end{corollary}
\begin{proof}
The proof follows similarly to the one of \Cref{thm:lowrankupdate} and thus is omitted.
\end{proof}

With the same assumptions as in~\Cref{thm:lowrankupdate}, we now
consider the  problem of finding the $\omega$-optimal
conditioning in the box $[0,1]^t$, i.e.,
\begin{equation}\label{eq:lrconst}
\bar \gamma = \arg \min_{\stackrel{\gamma\in{[0,1]^t}}{A(\gamma)\succ 0}}
        \omega(\gamma). 
\end{equation}
For the rank one update ($t=1$), \Cref{cor:rankoneconst} shows that 
the solution to~\Cref{eq:lrconst} can be obtained by first computing the
minimum of the unconstrained problem, whose explicit expression was
given in~\Cref{thm:optrankonegamma}, and then projecting onto the box
constraint, which in that case was the interval $[0,1]$. However, this
simple projection can fail in general for the low rank update, 
as we now show in~\Cref{ex:failproj} below. 

The illustration of this phenomenon will require
considering a constrained pseudoconvex minimization problem. 
In the following \Cref{thrm:pseudocnvprog}, see,
e.g.,~\cite[Chapter~10]{Mang:69}, we recall  the
sufficient optimality conditions for this class of optimization problems.
We note that no constraint qualification is needed for
\emph{sufficiency}.

\begin{fact}[Sufficient optimality conditions for pseudoconvex programming]
\label{thrm:pseudocnvprog}
Let $\Omega\subseteq \Rn$ be  nonempty open and convex. Let
$f:\Omega\to\R$ be  a pseudoconvex   function and
$(g_i)_{i=1}^m:\Omega\to\R$ a family of differentiable and quasiconvex
functions. Consider the optimization problem
\begin{equation}\label{eq:psudoconvprog}
\begin{array}{llr}
\min &f(x)  &\\
\text{\emph{s.t.}} & g_i(x) \leq 0,&  i\in[m], \\
& x\in\Omega. & 
\end{array}
\end{equation}
Let $\bar{x}\in \Omega, \bar{\lambda}\in\R^m$, be a KKT primal-dual
pair, i.e.,~the following \textdef{KKT conditions} hold:
\begin{equation}\label{eq:pseudoKKT}
\begin{aligned}
\nabla f(\bar{x}) + \sum_{i=1}^m \bar \lambda_i \nabla g_i(\bar x) & = 0  & \\
\bar \lambda_i &\geq 0, &  i\in[m], \\
\bar{\lambda}_i g_i(\bar x) & = 0,& i\in[m], \\
\bar{x}\in\Omega \text{ and }  g_i(\bar x)& \leq 0, & i\in[m].  \\
\end{aligned}
\end{equation}
Then $\bar x$ solves~\cref{eq:psudoconvprog}.
\end{fact}

\begin{example}[Failure of projection for constrained problem~\Cref{eq:lrconst}]
\label{ex:failproj}
Let $n=3$, $t=2$ and consider the following initial data for the 
$\omega$-minimization problem:
\begin{equation*}
A := \begin{bmatrix} 1 & 0 & 0 \\ 0 & 2 & 0 \\ 0 & 0 & 2 \end{bmatrix} \quad \text{ and } \quad U:= \begin{bmatrix} \frac {1}{\sqrt{2}} & 0 \\   \frac {-1}{\sqrt{2}} & 0 \\ 0 & 1 \end{bmatrix}.
\end{equation*}
Then, we get the following:
\begin{itemize}
\item 
From~\Cref{eq:optlr} and~\Cref{thm:lowrankupdate},
the $\omega$-optimal preconditioner is
$\gamma^* = \frac 13  \colvec{ \phantom{-} 1\\ -1 }$;
\item 
projecting onto $[0,1]^2$ yields
$\gamma^*_p =  \frac 13  \colvec{  1\\ 0 }$,
where $\omega(\gamma^*_p) = 16/(9 \sqrt[3]{5})$;

\item 
however, with $\bar \gamma := \frac 12  \colvec{  1\\ 0 }$,
we get a lower value:
\label{page:approxvalue}
\[
\omega(\bar\gamma) = 1/\left(3 \sqrt[3]{(2/11)^{2}}\right) \approx 
1.0386 < 1.0397 \approx 16/(9 \sqrt[3]{5}); 
\]
 and $\bar\gamma$ is the
$\omega$-optimal preconditioner  in $[0,1]^2$, as we now show.
\end{itemize}
To prove the last statement, note that~\Cref{eq:lrconst} can be written
as the pseudoconvex program in~\Cref{eq:psudoconvprog}  by setting
$f:=\omega:\R^2\to\R$, $g_1(\gamma) = - \gamma_1$,
$g_2(\gamma)= - \gamma_2$, $g_3(\gamma) = \gamma_1-1$,  $g_4(\gamma) =
\gamma_2-1$ and $\Omega$ defined as in~\Cref{eq:Omegalr}. In particular,
the only active constraint for $\bar \gamma = (1/2,0)^T$ is
$g_2(\gamma)=0$, so the KKT conditions become
\begin{equation*}
\begin{aligned}
0 &= \frac{\partial{\omega(\bar \gamma)}}{\partial \gamma_1},  \\
0 & = \frac{\partial{\omega(\bar \gamma)}}{\partial \gamma_2} - \bar{\lambda}_2,
\end{aligned}
\end{equation*}
for some $\bar{\lambda}_2 \geq 0$. This can be verified by  simply
substituting using the expressions of the partial derivatives of
$\omega$ obtained in~\cref{eq:partomega}.  By~\Cref{thrm:pseudocnvprog},
we conclude that for the given data, $\bar \gamma$ is the solution 
of~\Cref{eq:lrconst}.
\end{example}

As done in the previous example, obtaining the
$\omega$-optimal preconditioner in the box $[0,1]^t$  would require obtaining a
KKT point for the constrained pseudoconvex problem~\Cref{eq:lrconst}. This is
not an easy task. To the author's knowledge, closed formulas for this
kind of box constrained minimization problems are not known even when
the objective is a quadratic. Nevertheless, using the projection of
$\gamma^*$ onto $[0,1]^t$ as an approximation to $\bar \gamma$ appears
to give good results in practice. We see this in our numerical tests in~\Cref{section:NumericalTests}.

Finally, observe that  the computation of $\gamma^*$ in formula \Cref{eq:optgammalowrank} might be
as expensive as finding the Newton direction without preconditioning,
since it requires the spectral or Cholesky decomposition. However,
under the framework of \Cref{ex:generalizedJacobians}, 
we now see in \Cref{rem:gammaapprox} that
we can get the following inexpensive and effective approximation
$\gamma = \gamma^*_{\mathrm{apr}}$, given by the expression
\begin{equation}
\label{eq:estgamma}
\left[\gamma^*_{\mathrm{apr}}\right]_i:= \frac{\tr(A)}{(n-t)\|u_i\|^2}, \quad
\forall i\in [t].
\end{equation}

\begin{prop}
\label{rem:gammaapprox}
Let 
$A\in\Snp, B\in\R^{n\times m}, U\in\R^{n\times t}, 
\text{with } t:= |\cI_0|$,
be defined as in \Cref{ex:generalizedJacobians}. For simplicity
 define the columns of the matrix $\bar B:= \begin{bmatrix}B_i\end{bmatrix}_{i\in\cI_+}$
with $r:= \rank(\bar B)<n$. Let $j\in \cI_0$ and let
$u_j, w_j$ be defined as in \Cref{thm:lowrankupdate}, and
let $\epsilon> 0$ be the Levenberg--Marquardt regularization parameter 
in~\Cref{eq:GJ-setting}. Then,
	\[
u_j\notin\range(\bar B) \implies
	\frac{\|u_j\|^2}{\|w_j\|^2}\le \epsilon\frac{\|u_j\|^2}{\dist(u_j,\range(\bar B))^2}.
	\]
We conclude that \cref{eq:estgamma} provides an efficient estimate 
of $\gamma^*$ in formula \Cref{eq:optgammalowrank}.
\end{prop}
\begin{proof}
Let $x\in\range(\bar B)$, $y\in \nul(\bar B^T)$ such that $u_j = x+y$. Then,
	\[
	\begin{aligned}
		{\|w_j\|^2} &= u_j^TQD^{-1}Q^Tu_j\\
		&= (x+y)^TQD^{-1}Q^T(x+y)\\
		&= x^TQD^{-1}Q^Tx + 2x^TQD^{-1}Q^Ty + y^TQD^{-1}Q^Ty\\
		&= \sum_{\ell = 1}^{r}\frac{1}{\lambda_\ell + \epsilon}(q_\ell^Tx)^2 + 0 + \sum_{\ell = r+1}^{n}\frac{1}{\epsilon}(q_\ell^Ty)^2 \\
		&\ge \sum_{\ell = r+1}^{n}\frac{1}{\epsilon}(q_\ell^Ty)^2\\
		&= \frac{1}{\epsilon}\sum_{\ell=r+1}^{n}(q_\ell^Ty)^2 = \frac{1}{\epsilon}\sum_{\ell=1}^{n}(q_\ell^Ty)^2 = \frac{1}{\epsilon}\|y\|^2,
	\end{aligned}
	\]
	where $q_\ell$ is $\ell$-th column of $Q$. Since $\|y\|=\dist(u_j,\range(\bar B))$,
	\[
	\frac{\|u_j\|^2}{\|w_j\|^2}\le \epsilon\frac{\|u_j\|^2}{\dist(u_j,\range(\bar B))^2}.
	\]
\end{proof}

\begin{remark}
We note that the approximate $\omega$-optimal conditioning obtained in~\Cref{eq:estgamma} is strictly related to a popular choice of $\gamma$ appearing in the literature (see, e.g.,~\cite{CensorMoursiWeamsWolk:22}), namely, taking 
\begin{equation}\label{eq:optdiagforgenjac}
	\gamma_i:= \frac{1}{\|u_i\|^2},  \quad \forall i\in[t].
\end{equation}
Indeed, the updates resulting from~\Cref{eq:estgamma}   and~\Cref{eq:optdiagforgenjac} only differ in a scaling given by $\trace(A)/(n-t)$.
Recall from \Cref{prop:precond}~\Cref{item:optdiagprecond} that  the
selection of~\Cref{eq:optdiagforgenjac} corresponds to the
$\omega$-optimal diagonal preconditioner of the matrix $U$, i.e., it
aims to minimize the $\omega$-condition number of only the update term.
In contrast, our proposed approach aims to minimize  $\omega$ for the
whole matrix $A(\gamma) = A + U\Diag(\gamma)U^T$. Numerical comparisons
between both approaches are presented in the numerics in
\Cref{table:optgammapcg} and \Cref{fig:perfprofilespcg}, below.
\end{remark}




%
%

\section{Numerical Tests}
\label{section:NumericalTests}
We now present empirics for: the various preconditioners
\Cref{sect:empirprecond}; and the optimal preconditioned
low rank updates~\Cref{sect:numeriLowRank}.
The experiments were done on: Intel Core i7-12700H   2.30 GHz with 16GB
RAM, under Windows 11 (64-bit). We used  \textsc{Matlab} version 2024a.
The \textsc{Matlab} source code and data of all the experiments is available at \href{https://github.com/DavidTBelen/omega-condition-number.git}{https://github.com/DavidTBelen/omega-condition-number}.

\subsection{Comparisons of Preconditioners; Positive Definite Systems}
\label{sect:empirprecond}

In this section, we analyze the performance of an iterative method for
approximately solving  positive definite linear 
systems subject to different
preconditioning strategies. Specifically, we compare  the
$\omega$-optimal diagonal and
incomplete upper triangular $\omega$-optimal preconditioners introduced
above  with state-of-the-art preconditioners, e.g.,~the incomplete 
Cholesky preconditioner. Our test
enviroment follows the line of the extensive numerical comparisons
presented in the survey \cite{MR3638573}.

\subsubsection{Test Enviroment}
\label{subsect:test_enviroment}
The problems used in our experiment are all constructed with data from
the SuiteSparse Matrix Collection~\cite{Kolodziej2019}. We  consider the
symmetric positive definite matrices in this repository whose number of
rows (columns)  range from $5,000$ to $30,000$; but without
``duplicates'' (i.e., without similar matrices belonging to the same group). 
The right hand side of our linear system $b=e$, is always set as the 
vector of all ones. 

As the  iterative method for solving the positive definite linear systems,
we consider the implementation of the \textit{Preconditioned Conjugate
Gradients Method} given by \textsc{Matlab'}s built-in
function \textbf{pcg}. This is \textsc{Matlab}'s benchmark iterative solver for 
positive definite linear systems. In all our
experiments, our stopping criterion for \textbf{pcg} is when the
relative residual reaches a tolerance smaller than $10^{-6}$, i.e., 
\[
\frac{\|Wx-b\|}{\|b\|} < 10^{-6}.
\]

 Finally, in order to avoid ``trivialities'', we discard matrices
that generate problems that can be solved to the desired
tolerance in less than $10$ seconds by \textbf{pcg} with \emph{no
preconditioner}. This leaves a subset of $16$ matrices whose specific
characteristics are detailed in \Cref{tabl:exprecit}.
In the following we use $P$ to denote the set of these $16$ problems.

\subsubsection{Preconditioning Strategies}

We use the following strategies (with acronyms):

\begin{itemize}
\item \textbf{No preconditioning} (NONE).

\item The $\omega$-\textbf{optimal diagonal preconditioner} (DIAG)
given by \cref{eq:optdiagscal}.

\item The $\omega$-\textbf{optimal incomplete upper triangular
preconditioner} (ITRIU) given by \cref{eq:omegaoptincchol}. The dimension  $k$ of the triangular block is chosen according to the nonzero entries $nnz(W)$ of the matrix of interest $W$ as 
\index{$\lceil\cdot \rceil$, ceiling}
\index{ceiling, $\lceil\cdot \rceil$}
\[
k = \left\lceil \frac{1}{2} \left(1+\sqrt{1+\frac{4}{5} nnz(W)}\right) \right\rceil+1,
\]
where $\lceil\cdot \rceil$ is \emph{ceiling}.
The motivation on this choice resides in obtaining a
preconditioner with fewer nonzero entries than in $W$, i.e., $t(k-1) <<
nnz(W)$. The last summand $1$ ensures that the preconditioner is not
diagonal. 

\item \textbf{Incomplete Cholesky factorization} (ICHOL).  
\label{page:icholmatlab}
This preconditioning strategy consists in considering a Cholesky factorization of $W$, given by $LL^T$, but where some of the entries of $L$ are ignored agreeing with the sparsity pattern of $W$. The preconditioned system then becomes
\[
L^{-1} W L^{-T} y = L^{-1} b, \quad y = L^T x.
\]
We use \textsc{Matlab}'s \textbf{ichol} to construct $L$  and use the
options of the \textbf{pcg} solver for  solving the preconditioned
system without constructing $L^{-1}$ explicitly, as that could lead to
the loss of sparsity. To ensure that the process does not break down
(which can happen if a non positive pivot is encountered) we shift $W$
and obtain  an approximation of $W + \alpha\Diag(\diag W)$.
We make use of two different choices for the scaling factor~$\alpha$.
The first of them, denoted below as ICHOL(1) ,  uses the recommended
tuning in  the \textsc{Matlab} Help Center. However, as can be observed
in our tests below, this leads to a larger number of iterations of
\textbf{pcg} than what is expected from an incomple Cholesky
preconditioning. For instance, ICHOL(1) does not improve  ITRIU in this
aspect. Hence, we include a second choice of $\alpha$  taken two orders
of magnitude smaller than the recommended parameter in Matlab Help
Center. This significantly reduces the number of iterations required by \textbf{pcg}. In the experiments below, ICHOL(2) refers to this latter choice of $\alpha$.

\end{itemize}

\subsubsection{Performance Profile}
\label{subsect:performance_profile}

Besides illustrating the output from the experiments as displayed in
\Cref{tabl:exprecit,table:exp_prec_time,table:exp_prec_res,table:exp_prec_time_prec}, we 
 also employ performance profile plots, e.g.,~\cite{MR1875515}.
These plots are constructed as follows. Let 
$\Gamma:=\{\text{NONE, DIAG, ITRIU, ICHOL(1), ICHOL(2)}\}$ be the set of
preconditioners for our
comparisons. For each $p\in P$ and $\gamma \in\Gamma$, we denote as
$t_{p,\gamma}$ the measure we want to compare. In particular, we will
separately consider the number of iterations and the time required for
solving the system (to
the desired tolerance) for the preconditioned linear system
described in \Cref{subsect:test_enviroment}. In the cases where we consider a preconditioned system (i.e., all except NONE), the time for computing
the preconditioner is also included in $t_{p,\gamma}$, i.e.,
\[
\begin{aligned}
t_{p,\gamma} &=  \{\text{time for computing the preconditioner}\} \\
 &\quad + \{\text{time for solving  the  preconditioned problem by \textbf{pcg}}  \}.
\end{aligned}
\]
Then, for every problem $p\in P$ and every $\gamma \in \Gamma$, we define the performance ratio as
\begin{equation*}
r_{p,\gamma} :=
\left\{
\begin{array}{ll}
 \frac{ t_{p,\gamma}}{\min\{ t_{p,\gamma} \, : \, \gamma \in \Gamma\}} &\text{ if convergence test passed,} \\
 +\infty &\text{ if convergence test failed.}
\end{array}
\right.
\end{equation*}
In our experiments, a convergence test \emph{passed} if it succeeded in
solving the linear system with the required relative residual tolerance
in less than $100,000$ iterations, and otherwise it \emph{failed}. Note that the best performing preconditioner with respect to the measure under study (time or number of iterations), say $\tilde \gamma$, for problem $p$ will have performance ratio $r_{p,\tilde \gamma} = 1$. In contrast, if the preconditioner $\gamma$ underperforms in comparison with $\tilde\gamma$, but still manages to pass the test,  then 
\[
r_{p,\gamma} = \frac{ t_{p, \gamma}}{t_{p, \tilde{\gamma}}} > 1
\]
is the ratio between the overall time (resp., number of iterations) required for solving the problem $p$ for this particular choice and the time (resp., number of iterations) employed by $\tilde \gamma$. Consequently, the larger the value of $r_{p,\gamma}$, the worse the preconditioner $\gamma$ performed for problem $p$.

Finally, the performance profile of $\gamma \in \Gamma$ is defined as
\[
\rho_{\gamma}(\tau) := \frac{1}{|P|} \size\left\{ p \in P \, : \, r_{p, \gamma} \leq \tau \right\},
\]
where $|P|$ is the number of problems in $P$. This can be understood as
the relative portion of times that the performance ratio $r_{p,\gamma}$
is within a factor of $\tau \geq 1$  of the best possible performance
ratio. In particular, $\rho_{\gamma}(1)$ represents the number of
problems  where $\gamma$ is the best choice. Also, the existence of a $\tau\geq 1$ such that $\rho_{\gamma}(\tau) = 1$,  indicates that $\gamma$ passed the convergence test for every single problem in $P$.
In \Cref{fig:perf_profiles_prec}, we display our performance profiles,
with $\log_2$ scale on $\tau$.
\begin{figure}[ht!]\centering
 \includegraphics[height=.39\textwidth]{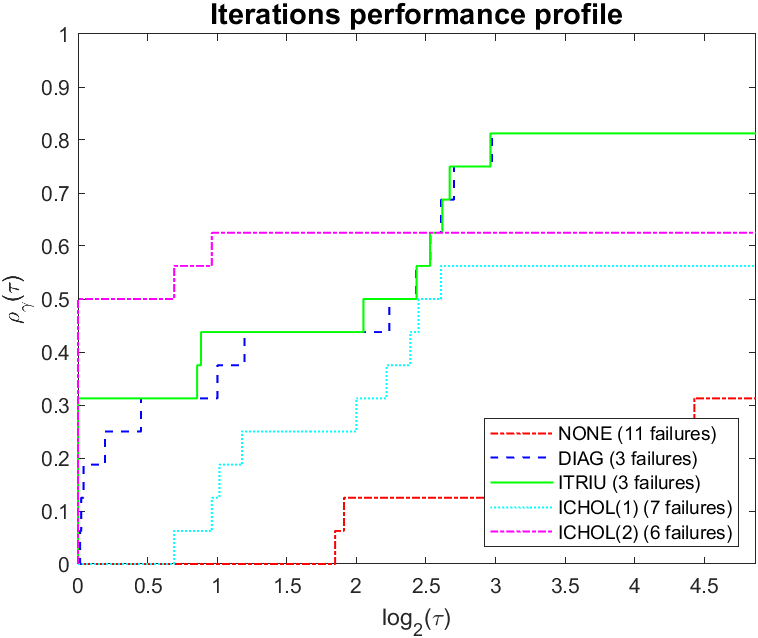}\hfill
 \includegraphics[height=.39\textwidth]{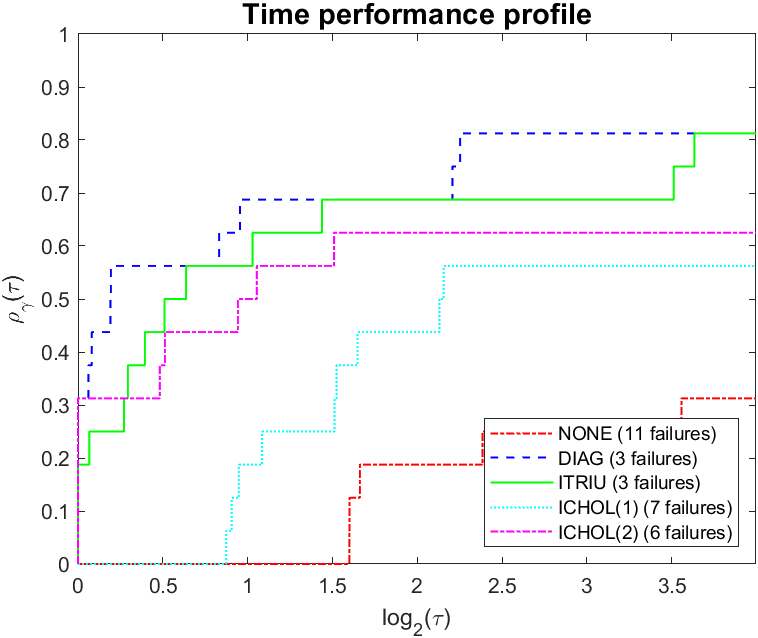}
 \caption{Iterations and time performance profiles for  solving the system with the different choices of preconditioner.}\label{fig:perf_profiles_prec}
\end{figure}

\subsubsection{Summary of the Empirics} 
\label{sect:summempirJac}

Our empirics suggest that the diagonal (DIAG) and the incomplete upper diagonal
$\omega$-optimal (ITRIU) preconditioners have  similar behaviour.   More precisely, ITRIU seems to  reduce   the number of iterations required by \textbf{pcg} in comparison to  DIAG (see \Cref{tabl:exprecit}).   This can be understood as a benefit of the additional reduction of the $\omega$-condition number furnished by the incomplete upper triangular block.
 In general, the incomplete Cholesky (ICHOL)(2)  appears to be the best solver for reducing the  number of iterations,  however it also fails to reach the desired relative residual accuracy more often (6 times) than the $\omega$-optimal preconditioners (3 times). The
residuals obtained by each one of the methods can be checked in \Cref{table:exp_prec_res}. 
Although the numerical results suggest that an appropriate incomplete Cholesky factorization provides a superior preconditioning, we note that the  $\omega$-optimal preconditioners  are more \emph{stable}, as the performance of the  incomplete Cholesky is usually influenced by the choice of the scaling factor $\alpha$. Indeed, a better result of \textbf{pcg} is obtained with the choice of  a small scaling factor $\alpha$ in ICHOL, but if $\alpha$ is too small a non positive pivot can be encountered and ICHOL may fail to be computable. In contrast, the $\omega$-preconditioners can always be computed without the need to manually set additional parameters.


\subsection{$\omega$-Optimal Low Rank Updates for Generalized Jacobians}
\label{sect:numeriLowRank}

We now present tests with different choices of $\gamma$ for efficient iterative
solutions of linear systems of the form \fbox{$A(\gamma) x= b$},
where $A(\gamma)$ is given in~\cref{eq:Aepsgamma}. 
We use \textsc{Matlab}'s builtin preconditioned conjugate gradient
function \textbf{pcg}. We focus
our attention on the case where $A(\gamma)\in \Snpp$ is a
low rank update that appears in choosing subgradients in nonsmooth Newton
methods, see~\Cref{ex:generalizedJacobians}. Our aim is to improve
conditioning to improve convergence, thus we call this 
\textdef{$\gamma$-conditioning}.

\subsubsection{Problem Generation and Definitions of $\gamma$}
\label{sect:probgenprec}
Specifically, we generate random instances as follows:
\begin{itemize}
\item  Define
\begin{equation}
\label{eq:Aepsgamma}
A(\gamma) := A +\epsilon I  + U \Diag(\gamma) U^T;
\end{equation}
\item
$\epsilon$ is a random   number in the interval $[10^{-7},10^{-9}]$;
\item 
$A= A_0^TA_0$ with $A_0\in\mathbb{R}^{r\times
n}$ a normally distributed random sparse matrix with density
at most $0.5/\log(n)$; $r\in [n/2+1,n-1]$ is a random integer;
\item 
$t \in [2,r/2]$ is the randomly chosen rank of the update,
$U\in\mathbb{R}^{n\times t}$ is a normally distributed random 
sparse matrix of density at most $1/\log(n)$;
\item The right hand side, $b$, is chosen as the sum of two random vectors in the range of $A$ and $U$, respectively. More precisely, 
\[
b = A \, b^1 + U\, b^2,
\]
with $b^1\in\R^n$ and $b^2\in\R^t$ vectors randomly generated using the standard normal distribution.
\end{itemize}
As explained in~\Cref{ex:generalizedJacobians}, in this application the 
$\gamma$ for conditioning is required to belong to the hypercube
$[0,1]^t$. Therefore, in our experiments we test  the performance of four
different choices of $\gamma$-conditioning:
\begin{itemize}
\item 
($\gamma=0$): the zero vector ;

\item ($\gamma=e$): the vector of ones;

\item  ($\gamma = u^{-2}$):
projection onto $[0,1]^t$ 
of the $\omega$-optimal diagonal preconditioner for the last term,
$U\Diag(\gamma)U^T$, of \Cref{eq:Aepsgamma}.
Recall from \Cref{prop:precond}, \Cref{item:optdiagprecond} that $\gamma$ is given by
\[
\gamma_i = \min\{ 1, 1/\|u_i\|^{2}\},\; i\in[t],
\]
where $u_i$ denotes the $i$th column of $U$; 
\item ($\gamma = \gamma^*_p$) projection  of $\gamma^*$, obtained
in~\Cref{thm:lowrankupdate}, onto $[0,1]^t$;
\item 
($\gamma = \gamma^*_{\mathrm{apr}}$):
projection of the approximated $\omega$-optimal 
obtained in~\Cref{eq:estgamma}, onto $[0,1]^t$.
\end{itemize}

\subsubsection{Descriptions of Parameters and Outputs}
For each dimension $n\in\{1000,2000,3000,5000\}$, we generate
$10$ instances of random problems and solve the corresponding systems 
with \textsc{Matlab}'s \textbf{pcg} and with 
the  five different choices of $\gamma$-conditioning.

\Cref{table:optgammapcg} shows the average over the $10$ instances of:
$\kappa$- and $\omega$-condition numbers of every $A(\gamma)$;
relative residual; number of iterations; time used by
\textbf{pcg} for every choice of $\gamma$; and time for computing each (nontrivial) $\gamma$. Both the spectral and Cholesky decompositions are implemented for computing $\gamma=\gamma^*_p$ but the table only contains the time for more efficient approach. We indicated the corresponding approach in the last column as well.
We stop if a tolerance of $10^{-12}$ is reached or the maximum
$50,000$ iterations is exceeded.
We use the origin as our initial starting point.
 
 
 \begin{table}[htp!]
 	\resizebox{\columnwidth}{!}{%
 		\input{tablejacobpcg.tex}}
 	\caption{For different dimensions $n$, every choice of $\gamma$ for
 		updating, average of $10$ instances: $\kappa$- and
 		$\omega$-condition numbers of $A(\gamma)$; residual; number
 		of iterations; total time (in seconds); solve time; time for computing $\gamma^*$,
 		(S) stands for spectral and (C) for Cholesky decomposition.}
 	\label{table:optgammapcg}
 \end{table}

We also use performance profiles to compare the different choices of $\gamma$; 
details in~\Cref{subsect:performance_profile}. 
Again, let $P$ denote the set of problems, and now set
$\Gamma:=\{0,e,u^{-2},\gamma^*_p,\gamma^*_{\mathrm{apr}}\}$ as the set of $\gamma$ conditioners.
We separately consider the number of iterations and the time required for
solving the system $A(\gamma) \, x = b$. We set the time
\[
t_{p,\gamma^*_p} = \{\text{time for solving  the system }  A(\gamma) \, x = b\} + \{\text{time for computing } \gamma^*_p\}.
\]
The latter quantity is taken as the minimum between the spectral and Cholesky approach.
For constructing the performance ratio in this setting, we consider that
a convergence test passed, rather than failed, if it succeeded in
solving the linear system with the required tolerance in less than
$50,000$ iterations. The output appears in~\Cref{fig:perfprofilespcg}.


\begin{figure}[ht!]\centering
 \includegraphics[height=.39\textwidth]{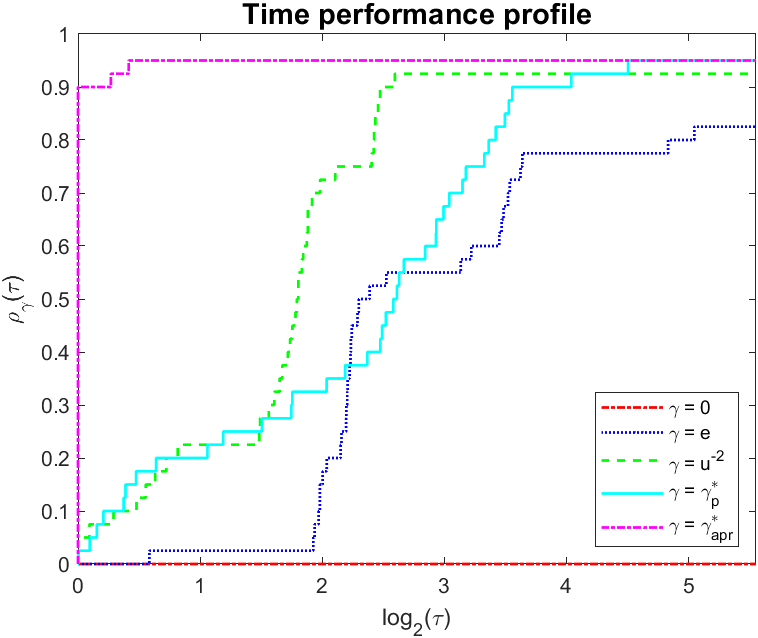}\hfill
 \includegraphics[height=.39\textwidth]{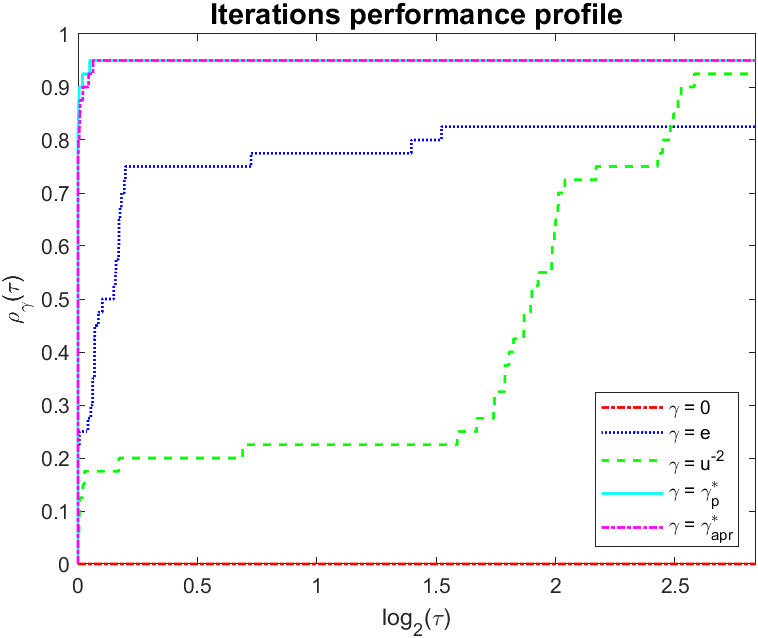}
 \caption{Time, iterations performance profiles; 
for system $A(\gamma)x=b$ with different choices
of $\gamma$ in \Cref{sect:probgenprec}; using  \textsc{Matlab}'s
\textbf{pcg}.}\label{fig:perfprofilespcg}
\end{figure}

\subsubsection{Summary of Empirics} 
\label{sect:summempir}

Firstly, we observe that \textbf{pcg} with no $\gamma$-conditioning ($\gamma{=}0$)  fails to achieve the desired residual, so the problems under consideration are sufficiently ill-conditioned. 
The performance profiles reveal that, in more than $90\%$ of the tested
instances,  the $\omega$-optimal conditioning leads to a problem that can
be solved with the least number of iterations. However, the computation of
the optimal conditioning $\gamma^*_p$ is too time expensive (see \Cref{table:optgammapcg}) which does not
make it advantageous in general\footnote{Regarding the different approaches for computing $\gamma=\gamma^*_p$, we want to mention that although obtaining the Cholesky
decomposition $A=LL^T$ is in general
less costly than computing its eigenvalue decomposition, the
computation of the $\omega$-optimal conditioning in this case requires
solving the system $LW = U$, see \Cref{cor:lrChol}. This means that, for larger dimensions,
employing the spectral decomposition for computing $\gamma^*$
seems to be more time efficient.}. Nonetheless, we observe that the approximated
 $\omega$-optimal conditioning $\gamma^*_{apr}$ maintains the benefits
 of  $\gamma^*_p$ in  terms of both solve time and number of iterations. In addition,
 it can be computed very efficiently which makes it the best option among all $\gamma$-conditionings.


\section{Conclusion}\label{sect:conclusion}
In this paper we have studied $\omega$, a nonclassical matrix 
condition number  formed as  the ratio of the arithmetic
and geometric means of eigenvalues. We have shown that $\omega$ has many
advantages over $\kappa$, the classic condition
number formed as the ratio of the largest to smallest eigenvalues, as
the latter is more of a \emph{worst case} condition number.
Moreover, the fact that $\kappa(A) = \kappa(A^{-1})$, (not true for
$\omega$) is misleading as
the conditioning of a linear system
$\cond(A) \neq \cond(A^{-1})$, and it is the latter that
in general needs reducing.  In fact, for linear systems $Ax=b$, we illustrated
empirically that it is $\omegatwo\cong \omega(A^{-2})$ 
that needs reducing. In addition, we found a new optimal 
diagonal preconditioner using this new measure, and we verified
its strengths empirically. This is currently of theoretical interest
only as exploiting $\omegatwo$ without evaluating
$A^{-1}$ first is still an open question.

We have used the differentiability and simplicity
of trace and determinant in $\omega(A)$ to find
optimal parameters for improving condition numbers for:
low rank updates that arise in the
context of  nonsmooth Newton methods; and for preconditioning for 
linear systems. We empirically show that the $\omega$-optimal preconditioners obtained in this work improve the performance of iterative methods. 
\label{page:positivetypo}

The $\omega$-condition
number, when compared to the classical $\kappa$-condition number,
is significantly more closely correlated to reducing the number of iterations
and time for iterative methods for positive definite linear systems.
This matches known results that show that preconditioning
for clustering of eigenvalues helps in iterative methods, i.e.,~using
all the eigenvalues rather than just the largest and smallest is desirable.
This is further evidenced by the empirics that show that $\omega(A)$ is
a significantly better estimate of 
the true conditioning of a linear system, i.e.,~how perturbations in the
data $A,b$ effect the solution $x$.

Finally, we have shown that an exact evaluation of $\omega(A)$ can be
found using either the Cholesky or LU factorization. This is in contrast
to the evaluation of $\kappa(A)$ that requires a spectral
decomposition or a $\|A\|\|A^{-1}\|$ evaluation.

In a future study we hope to continue on exploiting the measure
$\omega(A^{-2})$ by finding appropriate approximations,
e.g.,~\cite{WU2016828}.
The complexity of the denominator is unchanged, but
estimating the trace of an \emph{inverse} is a more difficult problem.
The condition number is also important in complexity analysis of
optimization methods, e.g.,~in the convergence of conjugate gradient
type methods.
We hope to avoid the worst case analysis to get a more average case
using $\omega$.

Finally we note that
the results we presented here can be extended beyond $A$ positive
definite by replacing eigenvalues with singular values in the definition
of $\omega(A)$.

\paragraph{Acknowledgements}
The authors would first like to thank Haesol Im and Walaa M. Moursi
for many useful and helpful conversations. We would also like to thank
two referees for many helpful comments that helped improve this paper.

\paragraph{Funding}
The author
D. Torregrosa-Bel\'en was partially supported by Centro de Modelamiento Matemático (CMM) BASAL fund FB210005 for center of excellence from ANID-Chile
and by Grants PGC2018-097960-B-C22
and PID2022-136399NB-C21 funded by ERDF/EU and by
MICIU/AEI/ 10.13039/501100011033.
Also by Grant PRE2019-090751 funded by
``ESF Investing in your future' and by MICIU/AEI/10.13039/501100011033.
\\All the authors were partially
supported by the National Research Council of Canada.

\paragraph{Data Availability} 
The \textsc{Matlab} source code and data of all the experiments in this manuscript are available at \href{https://github.com/DavidTBelen/omega-condition-number.git}{https://github.com/DavidTBelen/omega-condition-number}.

\section*{Declarations}

\paragraph{Conflict of interest} The authors declare they have no conflict of interest.

\newpage
\appendix

\section{Further Tables and $\omega$-Optimal Preconditioners}
\label{sect:furhterpreconditioners}
\label{sect:tablesapp}

\subsection{Tables}
We now present the tables for the empirics for the three
preconditioners in~\Cref{sect:empirprecond}. 
We use  matrices from the SuiteSparse Matrix Collection.

\begin{table}[h!]
\tiny
\resizebox{\columnwidth}{!}{%
\input{table_iterations.tex}}
\caption{preconditioners: number of iterations}
\label{tabl:exprecit}
\end{table}

\begin{table}[h!]
\tiny
\resizebox{\columnwidth}{!}{%
\input{table_time.tex}}
\caption{preconditioners: total time}
\label{table:exp_prec_time}
\end{table}

\begin{table}[h!]
\tiny
\resizebox{\columnwidth}{!}{%
\input{table_residuals.tex}}
\caption{preconditioners: residual $\|Wx-b\|$}
\label{table:exp_prec_res}
\end{table}

\begin{table}[htp!]
\tiny
\resizebox{\columnwidth}{!}{%
\input{table_time_prec.tex}}
\caption{Times (cpu) for computing the preconditioners}
\label{table:exp_prec_time_prec}
\end{table}

\subsection{$\omega$-Optimal Preconditioners}
\label{sect:furhterpreconditioners}

In this section we derive expressions for  $\omega$-optimal preconditioner matrices in different forms. The first one of them is a lower triangular two diagonal preconditioner. The second is a diagonal $+$ upper triangular preconditioner. The proofs of both results proceed similarly to Claim $1$ in \Cref{thm:incompleteuppertri}. Therefore, we will not reproduce the complete proofs and limit ourselves to highlight the main steps.

\subsection{Lower Triangular, Two Diagonal Preconditioning}
\label{sect:lowerTrTwoDiag}
In this section, we extend  the $\omega$-optimal diagonal scaling to an
$\omega$-optimal \textdef{lower triangular two diagonal scaling}.
We define \textdef{$\TDiag$} and 
\textdef{$\tdiag = \TDiag^*$} in obvious ways to construct the 
lower triangular two diagonal matrix from a vector and its adjoint.
Specifically, for a matrix $L=(L_{ij})_{i,j=1}^n \in \R^{n\times n}$, 
we get that
\[
\tdiag(L) = \begin{pmatrix}
         L_{1,1}\cr L_{2,2} \cr \ldots \cr L_{n,n} \cr
         L_{2,1}\cr L_{3,2} \cr L_{4,3}\cr  \ldots \cr L_{n,n-1} 
\end{pmatrix} =: \left( \begin{matrix} \phantom{1} \bar l \phantom{1} \cr \phantom{1} \hat l \phantom{1} \: \end{matrix}\right)\in \R^{n+(n-1)},
\]
while, given vectors $\bar{d} = (\bar{d}_1,\ldots,\bar{d}_n)^T \in \R^n $ and $\hat{d} = (\hat{d}_1,\ldots,\hat{d}_{n-1})\in\R^{n-1}$, we have
\[
\TDiag (\bar d,\hat d) = 
\begin{bmatrix}
\bar{d}_1 &0 & \ldots & \ldots   & \ldots &  0 \cr
\hat{d}_1& \bar{d}_2 & 0 & \ldots  &\ldots & 0 \cr
0 & \hat{d}_2 & \bar{d}_3& \vdots &  \vdots & 0  \cr
\vdots & \ldots & \ddots  & \ddots &  \vdots  & \vdots \cr
0 & \ldots &  \ldots& \hat{d}_{n-1} & \bar{d}_{n-1} & 0   \cr
0 &0& \ldots  &0 &  \hat{d}_{n-1} & \bar{d}_n   \cr
\end{bmatrix}.
\]
Note that $\TDiag :\R^{2n-1} \to \R^{n\times n}$ and
$\langle \TDiag(\bar d,\hat d), L\rangle = 
     \left\langle \dbardhat ,\tdiag(L)\right\rangle$, for any squared  matrix $L\in\R^{n\times n}$.

\begin{theorem}
\label{thm:Tdiags}
Let $W \in \Snpp$ and set
	\[
	\bar d^*_i = \begin{cases}
	\left(W_{i,i}-\frac{W_{i,i+1}^2}{W_{i+1,i+1}}\right)^{-1/2}
	=
\left(\frac{W_{i,i}W_{i+1,i+1}-W_{i,i+1}^2}{W_{i+1,i+1}}\right)^{-1/2},
& \text{if }i\in[n-1];\\
	W_{n,n}^{-1/2}, & \text{if }i=n
	\end{cases}
	\]
	and
	\[
\hat d^*_i  = 
        -\frac{W_{i,i+1}}{W_{i+1,i+1}}\bar d^*_i,\quad i\in [n-1].
	\]
Then the $\omega$-optimal lower triangular two diagonal
scaling of $W$ is given by
\begin{equation}
\label{eq:omegaobj}
(\bar d^*,\hat d^*) = \displaystyle \argmin_{(\bar d, \hat d) \in
\R^n_{++}
\times \R^{n-1}}   \omega (\bar d, \hat d), 
\end{equation}
where
\textdef{$\omega (\bar d, \hat d):= \omega\left(\TDiag(\bar d,\hat
d)^TW\TDiag(\bar d,\hat d)\right)$}.

\end{theorem}
\begin{proof}
First we note, since the $2\times 2$ principal minors for $W\succ 0$ are
all positive, the definitions of the optimal $d^*$ are well defined.
Let $\bar d\in \R^n_{++}$ and $\hat d\in\R^{n-1}$. Define the  
$\omega$-condition number, $f$ and $g$ as  functions of a pair $(\bar d,
\hat d) \in\R_{++}^n\times\R^{n-1}$. This is
\[
\omega(\bar d, \hat d) = \frac{f (\bar d,\hat d)}{g (d,\hat d)} := \frac{\tr\bigl(  \TDiag(\bar d,\hat d)^TW \TDiag(\bar
d,\hat d)\bigr)/n}{ \det (W)^{1/n}  \prod_{i=1}^n(\bar d_i)^{2/n} }.
\]
Differentiating the pseudoconvex $\omega$ and equating to $0$, we get the optimality condition 
	\begin{equation}\label{eq:2diagsoptcond}
		\left(\tdiag W\TDiag\right)(\bar d,\hat d) = \begin{pmatrix} \bar
			d^{-1} \cr 0_{n-1}\end{pmatrix}
	\end{equation}
Solving~\Cref{eq:2diagsoptcond} for $(\bar d, \hat d)$, results in 
	\[
	\bar d_i = \begin{cases}
		\left(W_{i,i}-\frac{W_{i,i+1}^2}{W_{i+1,i+1}}\right)^{-1/2}
		= \left(\frac{W_{i,i}W_{i+1,i+1}-W_{i,i+1}^2}{W_{i+1,i+1}}\right)^{-1/2}, & \text{if }i \in [n-1];\\
		W_{n,n} ^{-1/2}, & \text{if }i=n;
	\end{cases}
	\]
	and
	\[
\hat d_i  = -\frac{W_{i,i+1}}{W_{i+1,i+1}}\bar d_i,\quad i \in[n-1].
	\]

\end{proof}

\subsection{Upper Triangular $D_{+k}$ Diagonal Preconditioning}
\label{sect:optDponeprecondij}
We note that the $\omega$-optimal lower triangular two diagonal
preconditioner in~\Cref{thm:Tdiags} is sparse but its 
inverse though still lower triangular is not necessarily as sparse, 
i.e.,~the two diagonal structure can be lost completely, sparsity can be lost. 
We now consider the diagonal with upper
triangular elements that maintain the same structure in the inverse,
i.e.,~maintain sparsity for the inverse. Recall that the triangular number $t(k) = k(k+1)/2$ and define the transformation \textdef{$D_{+k}:\R^{n+t(k)} \to \R^{n\times n}$}:
\begin{equation}\label{eq:Dk}
	\begin{array}{rcl}
		D_{+k}(d,\alpha) &=&  \Diag(d) + 
		\begin{bmatrix}
			\begin{bmatrix}
				0_{n\times n\!-\!k}
			\end{bmatrix}
			&|&
			\begin{bmatrix}
				\begin{bmatrix}
					\Triu(\alpha)
				\end{bmatrix} \cr
				\begin{bmatrix}
					0_{n\!-\!k\times k}
				\end{bmatrix} 
			\end{bmatrix}
		\end{bmatrix}
		\vspace{.1in}
		\\ &=&
		\Diag(d) + \Triuk(\alpha) 
		=  \begin{bmatrix} \Diag & \Triuk\end{bmatrix}  
		\begin{pmatrix} d \cr  \alpha\end{pmatrix}
		\vspace{.1in}
		\\&=&
		\left(\begin{array}{ccccccccc}
			d_1 & 0 & \ldots & 0 & \ldots & \alpha_{1, n\!-\!k+1} & \alpha_{1, n\!-\!k+2} & \ldots & \alpha_{1, n}\\
			0 & d_2 & \ldots & 0 & \ldots & 0 & \alpha_{2, n\!-\!k+2} & \ldots & \alpha_{2, n}\\
			\vdots & \vdots & \ddots & \vdots & \ldots & \vdots & \vdots & \ddots & \vdots\\
			0 & 0 & \ldots & d_{k} & \ldots & 0 & 0 & 0 & \alpha_{k,n}\\
			\vdots & \vdots & \vdots & \vdots & \ddots & \vdots & \vdots & \vdots & \vdots\\
			0 & 0 & \ldots & 0 & \ldots & d_{n\!-\!k+1} & 0 & 0 & 0\\
			0 & 0 & \ldots & 0 & \ldots & 0 & d_{n\!-\!k+2} & 0 & 0\\
			\vdots & \vdots & \vdots & \vdots & \ldots & \vdots & \vdots & \ddots & \vdots\\
			0 & 0 & \ldots & 0 & \ldots & 0 & 0 & 0 & d_{n}
		\end{array}\right)
	\end{array}
\end{equation}
where $d \in \R^n$ and $\alpha:= (\alpha_{1, n\!-\!k+1}, \alpha_{1, n\!-\!k+2}, \alpha_{2, n\!-\!k+2}, \dots,\alpha_{1, n},\dots,\alpha_{k,n})^T \in \R^{t(k)}$.
Then the optimal upper triangular $D_{+k}(d,\alpha)$ diagonal preconditioner is given by solving the following optimization problem:
\begin{equation}
	\label{eq:uppertriag2diag}
(\bar d,\bar \alpha):= \argmin_{(d,\alpha)\in\R_{++}^n\times\R^{t(k)}}   
          \omega \big(D_{+k}(d,\alpha)^TWD_{+k}(d,\alpha)\big).
\end{equation}
\begin{theorem}\label{thm:Dplusk}
	Let $W\in\Snpp$ be given and let $(\bar d,\bar \alpha)\in \R^{n+{t(k)}}$ such that
	\begin{equation}\label{eq:diagD+k}
		\bar d_i = W_{i,i}^{-1/2},\quad  i\in[n-k]
	\end{equation}
	and the following hold for each $i\in [n-k+1,n]$:
	\begin{equation}\label{eq:optD+k}
		\begin{array}{rcll}
			W_{i,i}\bar d_i + \sum_{\ell=1}^{i-n+k}\bar \alpha_{\ell, i}W_{\ell,i} &=& 1/\bar d_i,\\
			W_{i,j}\bar d_i + \sum_{\ell=1}^{i-n+k}\bar \alpha_{\ell, i}W_{\ell,j} &=& 0,&j \in [i-n+k].
		\end{array}
	\end{equation}
	Then, $(\bar d, \bar \alpha)$ is the optimal solution of~\cref{eq:uppertriag2diag}.
\end{theorem}
\begin{proof}
Define the transformations (isometries)
\textdef{$\Triu : \R^{t(k)} \to \R^{k\times k}$} 
and 
\textdef{$\Triuk : \R^{t(k)} \to \R^{n\times n}$} according to \cref{eq:Dk}.
We denote the adjoints by \textdef{$\triu$} and
\textdef{$\triuk$}, respectively, and note that
\[
\triu^\dagger = \triu^*,\, \Triu^\dagger = \Triu^*.
\]
Hence,
\[
\begin{array}{rcl}
	D_{+k}(d,\alpha) &=& \Diag(d) + \Triuk(\alpha)\\
	&=& \begin{bmatrix} 
		\Diag & \Triuk 
	\end{bmatrix}\begin{pmatrix}
	d\\ \alpha
	\end{pmatrix}.
\end{array}
\]
Denote
\[
\begin{array}{rcl}
	\omega_k(d,\alpha) &:=&
	\omega\big(D_{+k}(d,\alpha)^TWD_{+k}(d,\alpha)\big)\\
	&=& \frac{\trace\big(D_{+k}(d,\alpha)^TWD_{+k}(d,\alpha)\big)/n}{\det\big(D_{+k}(d,\alpha)^TWD_{+k}(d,\alpha)\big)^{1/n}}\\
	&=& \frac{\trace\big(D_{+k}(d,\alpha)^TWD_{+k}(d,\alpha)\big)}{\det( W)^{1/n}\prod_{i=1}^{n}d_i^{2/n}}.
\end{array}
\]
For the numerator of $\omega_k$ we use
\[
\begin{array}{rcl}
	f(d,\alpha) 
	&:= & \frac{1}{n} \trace\big(D_{+k}(d,\alpha)^TWD_{+k}(d,\alpha)\big)\\
	&=& \frac{1}{n}
	\big\langle D_{+k}(d,\alpha), WD_{+k}(d,\alpha) \big\rangle\\
	&=& \frac{1}{n}
	\left\langle \begin{pmatrix}d \cr \alpha \end{pmatrix},  D_{+k}^* \big( W  D_{+k}(d,\alpha)\big) \right\rangle\\
	&=& \frac{1}{n}
	\begin{pmatrix}d \cr \alpha \end{pmatrix}^T  D_{+k}^* \big( W  D_{+k}(d,\alpha)\big)
	\\&=& \frac{1}{n}
	\begin{pmatrix}d \cr \alpha \end{pmatrix}^T \begin{bmatrix} \diag \cr \triuk\end{bmatrix} \big( W  D_{+k}(d,\alpha)\big)
	\\&=& \frac{1}{n}
	\begin{pmatrix}d \cr \alpha \end{pmatrix}^T \begin{bmatrix} 
		\diag W \big(\Diag(d) +\Triuk(\alpha) \big)
		\cr \triuk W \big(\Diag(d) +\Triuk(\alpha) \big)
	\end{bmatrix}
	\\&=& \frac{1}{n}
	\begin{pmatrix}d \cr \alpha \end{pmatrix}^T \begin{bmatrix} 
		\diag W \Diag &\diag W \Triuk  \cr
		\triuk W \Diag &\triuk W\Triuk
	\end{bmatrix}
	\begin{pmatrix}d \cr \alpha \end{pmatrix}.
\end{array}
\]
and the gradient is therefore
\begin{equation*}
	\nabla f(d,\alpha) 
	=
	\frac{2}{n} \begin{bmatrix} 
		\diag W \Diag &\diag W \Triuk  \cr
		\triuk W \Diag &\triuk W\Triuk
	\end{bmatrix}
	\begin{pmatrix}d \cr \alpha \end{pmatrix}.
\end{equation*}
The denominator of $\omega_k$ is
\[
g(d, \alpha):= \det(W)^{1/n}\prod_{i=1}^n d_i^{2/n}
\]
and thus
\[
\nabla g(d, \alpha) =
\frac 2n g(d, \alpha) \begin{pmatrix} 
	1/ d_1 \\ 1/ d_2 \\ \vdots \\ 1/ d_n
	\cr
	0 \\ \vdots \\0
\end{pmatrix}.
\]
For simplicity, denote $\bar d^{-1}:= (1/\bar d_1, 1/\bar d_2, \dots, 1/\bar d_n)^T\in \R^{n}$. Then,
\[
\begin{array}{rcl}
	\nabla \omega_k(d,\alpha) 
	&=& 
	\frac 1{g(d,\alpha)^2}  
	\big(g(d,\alpha) \nabla f(d,\alpha) - f(d,\alpha) \nabla g(d,\alpha)\big)
	\\&=&
	\frac {1}{g(d,\alpha)}  
	\left(\nabla f(d,\alpha) - \frac 2n f(d,\alpha) \begin{pmatrix}
		d^{-1} \cr 0_{t(k)}\end{pmatrix}
	\right).
\end{array}
\]
Finally, the proof follows from noticing that
\[
\begin{array}{rcl}
	(\bar d, \bar \alpha)\text{ satisfies~\Cref{eq:diagD+k,eq:optD+k}} &\iff& \frac{n}{2}\nabla f(\bar d, \bar \alpha)=\begin{pmatrix}
		\bar d^{-1}\cr 0_{t(k)}
	\end{pmatrix}\cr
	&\implies& f(\bar d, \bar \alpha) = 1.
\end{array}
\]
Hence, \Cref{eq:diagD+k,eq:optD+k} implies $\nabla \omega_k(\bar d,\bar \alpha)=0$, i.e., $(\bar d, \bar \alpha)$ is optimal.

\end{proof}

The following~\Cref{ex:D+1} and \Cref{ex:D+2} solve~\Cref{eq:optD+k} for $k=1$ and $k=2$.
\begin{example}[$k=1$]
	\label{ex:D+1}
	Let $W \in \Snpp$ be given. Set
	\[
	\bar d_i = 
	\begin{cases}
		W_{i,i}^{-1/2},& \text{if }i \in [n-1]\\
		\left(\frac{W_{1,1}W_{n,n}-W_{1,n}^2}{W_{1,1}}\right)^{-1/2},&\text{if }i=n.
	\end{cases}
	\]
	and
	\[
	\bar \alpha = 
	-\frac{W_{1n}}{W_{11}}\bar d_n.
	\]
	Then the optimal $D_{+1}$-diagonal upper triangular scaling is given by
	\begin{equation*}
		(\bar d,\bar \alpha) = \displaystyle \argmin_{ d\in
\R_{++}^n, \alpha \in \R}
		\omega \big(D_{+1}(d,\alpha)^TWD_{+1}(d,\alpha)\big).
	\end{equation*}
\end{example}

\begin{example}[$k=2$]
	\label{ex:D+2}
	Let $W \in \Snpp$ be given. Set
	\[
	\bar d_i = 
	\begin{cases}
		W_{i,i}^{-1/2},& \text{if }i \in [n-2]\\
		\left(\frac{W_{1,1}W_{n-1,n-1}-W_{1,n-1}^2}{W_{1,1}}\right)^{-1/2},&\text{if }i=n-1\\
		\left(W_{n,n} + \frac{W_{1,n}^2W_{2,2} - 2W_{1,n}W_{2,n}W_{1,2} + W_{2,n}^2W_{1,1}}{W_{1,2}^2 - W_{1,1}W_{2,2}}\right)^{-1/2},&\text{if }i=n.
	\end{cases}
	\]
	\[
	\begin{array}{rcl}
		\bar \alpha_{1,n} &=& \left(\frac{W_{1,n}W_{2,2}-W_{1,2}W_{2,n}}{W_{1,2}^2-W_{1,1}W_{2,2}}\right)\bar d_n,\\
		\bar\alpha_{1, n-1} &=& -\frac{W_{1,n-1}}{W_{1,1}}\bar d_{n-1},\\
		\bar \alpha_{2,n} &=& \left(\frac{W_{1,1}W_{2,n}-W_{1,2}W_{1,n}}{W_{1,2}^2-W_{1,1}W_{2,2}}\right)\bar d_n.
	\end{array}
	\]
	Then the optimal $D_{+2}$-diagonal upper triangular scaling is given by
	\begin{equation*}
		(\bar d,\bar \alpha) = \displaystyle \argmin_{ d\in
\R_{++}^n, \alpha \in \R^3}
		\omega \big(D_{+2}(d,\alpha)^TWD_{+2}(d,\alpha)\big).
	\end{equation*}
\end{example}

\cleardoublepage
\phantomsection
\addcontentsline{toc}{section}{Index}
\printindex
\label{ind:index}

\cleardoublepage
\phantomsection
\addcontentsline{toc}{section}{Bibliography}
\bibliographystyle{siam}
\bibliography{.master,.edm,.psd,.bjorBOOK,davidt,leo}
\label{bib:bibl}

\end{document}

%% file: tableomegatimes.tex
\begin{tabular}{|c|c|c|c|c|c|c|c|c|c|} 
 \hline 
$n$ & Fact.  & order $\kappa$ 1e2& order $\kappa$ 1e3& order $\kappa$ 1e4& order $\kappa$ 1e5& order $\kappa$ 1e6& order $\kappa$ 1e7& order $\kappa$ 1e8& order $\kappa$ 1e9\\ 
 \hline 
\multirow{3}{*}{500} & eig & 5.5267e-02& 5.7766e-02& 5.2747e-02& 5.9256e-02& 6.0856e-02& 6.2197e-02& 5.5592e-02& 5.7626e-02\\ 
 \cline{2-10} 
 & R & 1.1218e-02& 8.0907e-03& 7.5172e-03& 8.4705e-03& 9.2774e-03& 8.5553e-03& 8.1462e-03& 7.9027e-03\\ 
 \cline{2-10} 
 & LU & 2.2893e-02& 1.8159e-02& 1.8910e-02& 2.0902e-02& 2.0057e-02& 2.0308e-02& 1.9060e-02& 1.8879e-02\\ 
 \hline 
 \multirow{3}{*}{1000} & eig & 3.0664e-01& 2.8968e-01& 2.6095e-01& 2.7796e-01& 5.7083e-01& 5.9007e-01& 5.8351e-01& 5.9630e-01\\ 
 \cline{2-10} 
 & R & 2.9328e-02& 2.8339e-02& 2.7869e-02& 3.1909e-02& 5.8628e-02& 6.0873e-02& 6.2429e-02& 6.1074e-02\\ 
 \cline{2-10} 
 & LU & 7.5011e-02& 7.2666e-02& 7.0497e-02& 7.6778e-02& 1.6313e-01& 1.7313e-01& 1.7666e-01& 1.7326e-01\\ 
 \hline 
 \multirow{3}{*}{2000} & eig & 3.4794e+00& 3.4804e+00& 3.1916e+00& 3.4386e+00& 3.4235e+00& 3.4766e+00& 3.2327e+00& 3.3704e+00\\ 
 \cline{2-10} 
 & R & 3.5644e-01& 3.5989e-01& 2.9556e-01& 3.6375e-01& 3.5847e-01& 3.5972e-01& 3.2629e-01& 3.4227e-01\\ 
 \cline{2-10} 
 & LU & 9.0136e-01& 9.0537e-01& 7.1161e-01& 8.7445e-01& 8.6420e-01& 8.8027e-01& 8.1990e-01& 8.1383e-01\\ 
 \hline 
 \end{tabular}

%% file: tableomegaprec.tex
\begin{tabular}{|c|c|c|c|c|c|c|c|c|c|} 
 \hline 
$n$ & Fact.  & order $\kappa$ 1e2& order $\kappa$ 1e3& order $\kappa$ 1e4& order $\kappa$ 1e5& order $\kappa$ 1e6& order $\kappa$ 1e7& order $\kappa$ 1e8& order $\kappa$ 1e9\\ 
 \hline 
\multirow{3}{*}{500} & eig & 1.5632e-13& 2.7853e-12& 2.2618e-10& 1.2695e-08& 8.9169e-07& 5.4109e-05& 2.2610e-03& 1.7349e-01\\ 
 \cline{2-10} 
 & R & 1.7053e-13& 2.5580e-12& 1.0039e-10& 1.1339e-08& 4.9818e-07& 2.6470e-05& 1.3173e-03& 1.6217e-01\\ 
 \cline{2-10} 
 & LU & 1.5987e-13& 2.4585e-12& 1.0652e-10& 1.1987e-08& 5.1592e-07& 2.1372e-05& 1.3641e-03& 1.4268e-01\\ 
 \hline 
 \multirow{3}{*}{1000} & eig & 2.1316e-13& 2.1032e-12& 8.7653e-11& 4.6271e-09& 3.1477e-07& 1.9602e-05& 9.9290e-04& 7.6469e-02\\ 
 \cline{2-10} 
 & R & 4.2633e-13& 1.5632e-12& 4.2235e-11& 3.9297e-09& 2.9562e-07& 1.1498e-05& 9.1506e-04& 5.3287e-02\\ 
 \cline{2-10} 
 & LU & 4.4054e-13& 1.4850e-12& 3.7858e-11& 3.8287e-09& 2.7390e-07& 1.3820e-05& 6.0492e-04& 4.8568e-02\\ 
 \hline 
 \multirow{3}{*}{2000} & eig & 2.4336e-13& 4.1780e-12& 4.2019e-10& 2.0080e-08& 7.7358e-07& 6.4819e-05& 5.5339e-03& 3.7527e-01\\ 
 \cline{2-10} 
 & R & 4.3698e-13& 2.0819e-12& 5.0704e-11& 2.3442e-09& 1.8376e-07& 8.9575e-06& 5.5255e-04& 4.8842e-02\\ 
 \cline{2-10} 
 & LU & 4.3165e-13& 2.2595e-12& 2.3249e-11& 2.5057e-09& 1.5020e-07& 6.0479e-06& 5.4228e-04& 4.4205e-02\\ 
 \hline 
 \end{tabular}

%% file: tablejacobpcg.tex
\begin{tabular}{|c|c|c|c|c|c|c|c|c|} \hline
$n$ & $\gamma$ & $\kappa(A(\gamma))$ & $\omega(A(\gamma))$ & Rel.Res & Iter & T. Total & T. Solve & T. $\gamma^*$ \\
  \hline 
\multirow{4}{*}{1000} & 0 &2.3249e+09 & 1.0173e+04 & 1.5146e-01 & 2734.30 &  0.0218 & 0.0218 & -\\ 
 \cline{2-9} 
& e &2.2193e+11 & 4.7689e+03 & 1.7750e-06 & 4187.40 &  0.8181 &   0.8181 & -\\ 
 \cline{2-9} 
& $u^{-2}$ &8.5539e+09 & 2.2114e+03 & 3.8172e-07 & 2730.10 &  0.0810 & 0.0807 & 0.0003\\ 
 \cline{2-9} 
& $\gamma^{\ast}_p$ &1.0722e+10 & 2.2095e+03 & 1.8169e-07 & 2788.10 &  0.0880 & 0.0748 & 0.0132 (C)\\ 
 \cline{2-9} 
& $\gamma^{\ast}_{\text{apr}}$ &1.0432e+10 & 2.2095e+03 & 1.5855e-07 & 2821.80 &  0.0757 & 0.0752 & 0.0005\\ 
 \hline 
 \multirow{4}{*}{2000} & 0 &2.0729e+12 & 2.2127e+04 & 1.8829e-07 & 230.70 &  0.5887 & 0.5887 & -\\ 
 \cline{2-9} 
& e &3.4235e+12 & 6.5292e+02 & 9.2055e-13 & 425.50 &  1.4527 & 1.4527 & -\\ 
 \cline{2-9} 
& $u^{-2}$ &1.8491e+12 & 8.9700e+02 & 9.1536e-13 & 1364.40 &  0.7545 & 0.7541 & 0.0003\\ 
 \cline{2-9} 
& $\gamma^{\ast}_p$ &2.0726e+12 & 5.6538e+02 & 9.1902e-13 & 376.20 &  0.6021 & 0.2319 & 0.3702 (S)\\ 
 \cline{2-9} 
& $\gamma^{\ast}_{\text{apr}}$ &2.0644e+12 & 5.6538e+02 & 9.1737e-13 & 376.30 &  0.2391 & 0.2368 & 0.0023\\ 
 \hline 
 \multirow{4}{*}{3000} & 0 &4.4961e+12 & 1.9718e+04 & 7.4824e-08 & 586.20 &  3.3928 & 3.3928 & -\\ 
 \cline{2-9} 
& e &3.6078e+12 & 3.9795e+03 & 9.3498e-13 & 261.70 &  1.5414 & 1.5414 & -\\ 
 \cline{2-9} 
& $u^{-2}$ &3.6699e+12 & 4.3747e+03 & 9.1465e-13 & 917.80 &  1.0956 & 1.0954 & 0.0002\\ 
 \cline{2-9} 
& $\gamma^{\ast}_p$ &3.6544e+12 & 3.9795e+03 & 9.4326e-13 & 261.60 &  1.6990 & 0.3219 & 1.3770 (S)\\ 
 \cline{2-9} 
& $\gamma^{\ast}_{\text{apr}}$ &3.5688e+12 & 3.9795e+03 & 9.4326e-13 & 261.60 &  0.3247 & 0.3199 & 0.0048\\ 
 \hline 
 \multirow{4}{*}{5000} & 0 &1.0028e+13 & 3.0421e+04 & 2.6256e-07 & 698.80 &  11.5600 & 11.5600 & -\\ 
 \cline{2-9} 
& e &1.1709e+13 & 8.9242e+02 & 9.3629e-13 & 362.90 &  5.9142 & 5.9142 & -\\ 
 \cline{2-9} 
& $u^{-2}$ &8.2778e+12 & 1.4249e+03 & 1.2583e-09 & 1563.00 &  6.3804 & 6.3783 & 0.0021\\ 
 \cline{2-9} 
& $\gamma^{\ast}_p$ &8.7440e+12 & 8.2755e+02 & 9.6946e-13 & 344.30 &  8.3604 & 1.4008 & 6.9596 (S)\\ 
 \cline{2-9} 
& $\gamma^{\ast}_{\text{apr}}$ &8.8089e+12 & 8.2755e+02 & 9.5456e-13 & 344.30 &  1.4175 & 1.4026 & 0.0149\\ 
 \hline 
 \end{tabular}

%% file: table_iterations.tex
\begin{tabular}{|l|rr|r|r|r|r|r|} \hline
name & $n$ & $nnz(W)$ & NONE & DIAG & ITRIU & ICHOL(1) & ICHOL(2)  \\
  \hline 
\verb|mhd4800b| &4800 & 27520  & $>$97633  & 26  & 19  & 37  & 37 \\ 
\verb|s3rmt3m3| &5357 & 207123  & $>$99172  & 14283  & 14134  & $>$15207  & $>$14300 \\ 
\verb|ex15| &6867 & 98671  & $>$98296  & 46299  & 45029  & $>$99102  & $>$47275 \\ 
\verb|bcsstk38| &8032 & 355460 & - & 10104  & 8837  & 14264  & 14267 \\ 
\verb|aft01| &8205 & 125567  & $>$8452  & 786  & 780  & 610  & 100 \\ 
\verb|nd3k| &9000 & 3279690  & 6012  & 9245  & 8632  & $>$8007  & 1599 \\ 
\verb|bloweybq| &10001 & 49999 & -& - & $>$11 & -& -\\ 
\verb|msc10848| &10848 & 1229776  & 56719  & 5274  & 4767  & 5328  & 2634 \\ 
\verb|t2dah_e| &11445 & 176117  & $>$99495  & 33  & 29  & 28  & 7 \\ 
\verb|olafu| &16146 & 1015156  & $>$90196  & 28028  & 22572  & 27670  & 12232 \\ 
\verb|gyro| &17361 & 1021159  & 28942  & 11605  & 11684  & 9964  & 1904 \\ 
\verb|nd6k| &18000 & 6897316  & 6589  & 9857  & 10574  & 8515  & 1831 \\ 
\verb|raefsky4| &19779 & 1316789 & - & 82865  & 81551  & $>$87212  & $>$17990 \\ 
\verb|LFAT5000| &19994 & 79966 & - & $>$4984  & $>$5037 & -& -\\ 
\verb|msc23052| &23052 & 1142686 & - & $>$91699  & $>$91700  & $>$99722  & $>$98530 \\ 
\verb|smt| &25710 & 3749582  & 9764  & 3343  & 3273  & 2803  & 514 \\ 
 \hline 
 \end{tabular}

%% file: table_time.tex
\begin{tabular}{|l|rr|r|r|r|r|r|} \hline
name & $n$ & $nnz(W)$ & NONE & DIAG & ITRIU & ICHOL(1) & ICHOL(2)  \\
  \hline 
\verb|mhd4800b| &4800 & 27520  & $>$2.70  & 0.00  & 0.00  & 0.01  & 0.00 \\ 
\verb|s3rmt3m3| &5357 & 207123  & $>$8.54  & 1.76  & 2.16  & $>$2.07  & $>$2.75 \\ 
\verb|ex15| &6867 & 98671  & $>$6.54  & 2.63  & 2.76  & $>$14.68  & $>$11.94 \\ 
\verb|bcsstk38| &8032 & 355460  & $>$18.72  & 1.94  & 1.83  & 3.36  & 5.22 \\ 
\verb|aft01| &8205 & 125567  & $>$0.66  & 0.07  & 0.07  & 0.15  & 0.03 \\ 
\verb|nd3k| &9000 & 3279690  & 13.13  & 19.78  & 47.35  & $>$19.50  & 4.15 \\ 
\verb|bloweybq| &10001 & 49999  & $>$5.41  & $>$5.35  & $>$5.97  & $>$25.40  & $>$24.97 \\ 
\verb|msc10848| &10848 & 1229776  & 38.88  & 3.78  & 3.30  & 6.19  & 6.35 \\ 
\verb|t2dah_e| &11445 & 176117  & $>$12.74  & 0.01  & 0.01  & 0.02  & 0.01 \\ 
\verb|olafu| &16146 & 1015156  & $>$59.21  & 16.55  & 15.86  & 30.59  & 22.62 \\ 
\verb|gyro| &17361 & 1021159  & 19.61  & 7.39  & 8.50  & 18.61  & 6.47 \\ 
\verb|nd6k| &18000 & 6897316  & 30.31  & 46.16  & 124.02  & 43.77  & 10.00 \\ 
\verb|raefsky4| &19779 & 1316789  & $>$75.59  & 62.72  & 89.31  & $>$223.45  & $>$68.90 \\ 
\verb|LFAT5000| &19994 & 79966  & $>$10.08  & $>$10.16  & $>$10.27  & $>$24.52  & $>$25.49 \\ 
\verb|msc23052| &23052 & 1142686  & $>$78.53  & $>$77.28  & $>$79.37  & $>$106.01  & $>$173.72 \\ 
\verb|smt| &25710 & 3749582  & 25.31  & 8.63  & 13.13  & 15.18  & 4.85 \\ 
 \hline 
 \end{tabular}

%% file: table_residuals.tex
\begin{tabular}{|l|rr|r|r|r|r|r|} \hline
name & $n$ & $nnz(W)$ & NONE & DIAG & ITRIU & ICHOL(1) & ICHOL(2)  \\
  \hline 
\verb|mhd4800b| &4800 & 27520  & -  & 5.860e-02  & 7.919e-02  & 5.221e-05  & 4.173e-05 \\ 
\verb|s3rmt3m3| &5357 & 207123  & -  & 5.149e-02  & 4.339e-02  & -  & - \\ 
\verb|ex15| &6867 & 98671  & -  & 2.327e+00  & 2.335e+00  & -  & - \\ 
\verb|bcsstk38| &8032 & 355460  & -  & 1.148e-01  & 6.272e-02  & 8.006e-05  & 7.169e-05 \\ 
\verb|aft01| &8205 & 125567  & -  & 2.083e-04  & 3.355e-04  & 8.085e-05  & 5.592e-05 \\ 
\verb|nd3k| &9000 & 3279690  & 9.084e-05  & 1.089e-04  & 5.620e-04  & -  & 8.861e-05 \\ 
\verb|bloweybq| &10001 & 49999  & -  & -  & -  & -  & - \\ 
\verb|msc10848| &10848 & 1229776  & 8.705e-05  & 2.156e-03  & 1.563e-03  & 1.002e-04  & 7.700e-05 \\ 
\verb|t2dah_e| &11445 & 176117  & -  & 1.338e-04  & 1.784e-03  & 8.938e-05  & 8.211e-05 \\ 
\verb|olafu| &16146 & 1015156  & -  & 2.182e-03  & 7.755e-04  & 1.118e-04  & 9.864e-05 \\ 
\verb|gyro| &17361 & 1021159  & 1.289e-04  & 1.955e-04  & 2.167e-04  & 1.234e-04  & 1.174e-04 \\ 
\verb|nd6k| &18000 & 6897316  & 1.324e-04  & 1.602e-04  & 7.317e-04  & 1.325e-04  & 1.327e-04 \\ 
\verb|raefsky4| &19779 & 1316789  & -  & 2.273e-01  & 2.573e-01  & -  & - \\ 
\verb|LFAT5000| &19994 & 79966  & -  & -  & -  & -  & - \\ 
\verb|msc23052| &23052 & 1142686  & -  & -  & -  & -  & - \\ 
\verb|smt| &25710 & 3749582  & 1.450e-04  & 2.729e-04  & 2.658e-04  & 1.530e-04  & 1.310e-04 \\ 
 \hline 
 \end{tabular}

%% file: table_time_prec.tex
\begin{tabular}{|l|rr|r|r|r|r|} \hline
name & $n$ & $nnz(W)$  & DIAG & ITRIU & ICHOL(1) & ICHOL(2) \\
  \hline 
\verb|mhd4800b| &4800 & 27520  & 1.089e-03  & 3.193e-03  & 2.079e-03  & 7.248e-04 \\ 
\verb|s3rmt3m3| &5357 & 207123  & 6.965e-04  & 2.426e-03  & 1.644e-03  & 1.831e-03 \\ 
\verb|ex15| &6867 & 98671  & 4.206e-04  & 1.234e-03  & 9.327e-04  & 2.310e-03 \\ 
\verb|bcsstk38| &8032 & 355460  & 7.168e-04  & 5.815e-03  & 2.181e-03  & 2.136e-03 \\ 
\verb|aft01| &8205 & 125567  & 5.021e-04  & 2.030e-03  & 1.146e-03  & 1.925e-03 \\ 
\verb|nd3k| &9000 & 3279690  & 2.506e-03  & 7.995e-02  & 1.712e-02  & 1.676e-02 \\ 
\verb|bloweybq| &10001 & 49999  & 5.385e-04  & 1.365e-03  & 4.774e-04  & 4.115e-04 \\ 
\verb|msc10848| &10848 & 1229776  & 1.481e-03  & 8.259e-03  & 8.472e-03  & 2.943e-02 \\ 
\verb|t2dah_e| &11445 & 176117  & 7.269e-04  & 2.055e-03  & 2.599e-03  & 5.342e-03 \\ 
\verb|olafu| &16146 & 1015156  & 1.789e-03  & 1.623e-02  & 6.528e-03  & 1.793e-02 \\ 
\verb|gyro| &17361 & 1021159  & 1.730e-03  & 1.107e-02  & 1.736e-02  & 6.125e-02 \\ 
\verb|nd6k| &18000 & 6897316  & 5.121e-03  & 2.661e-01  & 3.337e-02  & 3.741e-02 \\ 
\verb|raefsky4| &19779 & 1316789  & 1.755e-03  & 2.275e-02  & 2.149e-02  & 6.287e-02 \\ 
\verb|LFAT5000| &19994 & 79966  & 8.536e-04  & 1.299e-03  & 8.915e-04  & 1.077e-03 \\ 
\verb|msc23052| &23052 & 1142686  & 2.056e-03  & 7.827e-03  & 5.936e-03  & 1.119e-02 \\ 
\verb|smt| &25710 & 3749582  & 3.886e-03  & 1.172e-01  & 5.734e-02  & 2.481e-01 \\ 
 \hline 
 \end{tabular}